\documentclass[a4paper, 10pt]{article} \usepackage[utf8]{inputenc} \usepackage{mathtools} \usepackage[T2A]{fontenc} \usepackage{hyperref} \usepackage{dsfont} \usepackage[T2A]{fontenc} \usepackage{amsmath,amssymb,amsthm} \usepackage[a4paper, lmargin = 3cm, rmargin = 3cm, tmargin = 3cm, bmargin = 2cm]
{geometry} \usepackage[russian, english]{babel} 
\usepackage{xcolor}
\usepackage{tikz-cd} 
\usepackage{cleveref}
\usepackage{multicol}
\usepackage{lipsum}
\usepackage{mwe}
\usepackage{subcaption}
\usepackage[colorinlistoftodos]{todonotes}
\crefformat{footnote}{#2\footnotemark[#1]#3}

\usepackage{relsize}

\hypersetup{
    colorlinks,
    linkcolor={blue},
    citecolor={blue},
    urlcolor={blue}
}

\newcommand\blfootnote[1]{%
  \begingroup
  \renewcommand\thefootnote{}\footnote{#1}%
  \addtocounter{footnote}{-1}%
  \endgroup
}

\newtheorem{lemma}{Lemma}
\numberwithin{lemma}{section}
\newtheorem{theorem}{Theorem}

\newtheorem{prop}{Proposition}

\newtheorem{defin}{Definition}
\newtheorem{cor}{Corollary}
\newtheorem{rem}{Remark}

\newtheorem*{hypo*}{Conjecture}

\title{Erdős inequality for primitive sets}
\author{Petr Kucheriaviy}
\date{}

\begin{document}

\maketitle

\blfootnote{The author was supported by the Basic Research Program of the National Research University Higher School of Economics.}

\begin{abstract}
A set of natural numbers $A$ is called primitive if no element of $A$ divides any other. 
Let $\Omega(n)$ be the number of prime divisors of $n$ counted with multiplicity. 
Let $f_z(A) = \sum_{a \in A}\frac{z^{\Omega(a)}}{a (\log a)^z}$, where $z \in \mathbb{R}_{> 0}$. 
Erdős proved in 1935 that $f_1(A) = \sum_{a \in A}\frac{1}{a \log a}$ is uniformly bounded over all choices of primitive sets $A$. We prove the same fact for $f_z(A)$, when $z \in (0, 2)$. 
Also we discuss the $\lim_{z \to 0} f_z(A)$.
Some other results about primitive sets are generalized. 
In particular we study the asymptotic of $f_z(\mathbb{P}_k)$, where $\mathbb{P}_k = \{ n : \Omega(n) = k \}$. In case of $z = 1$ we find the next term in asymptotic expansion of $f_1(\mathbb{P}_k)$ compared to the recent result of Gorodetsky, Lichtman, Wong.
\end{abstract}

\tableofcontents

\section{Introduction and results}

\subsection{Erdős inequality}

Let $S$ be a partially ordered set. A subset $A \subset S$ is called an antichain if every two distinct elements of $A$ are incomparable. 

By primitive set we will mean an antichain of the set of natural numbers equipped with the relation of divisibility.

So by definition $A \subset \mathbb{N}$ is called primitive if for any $a_1, a_2 \in A$ such that $a_1|a_2$, we have $a_1 = a_2$.

$A = [n, 2n)$ is an example of a primitive set.
We denote by $\omega(n)$ the number of distinct prime divisors of $n$. And $\Omega(n)$ stands for the number of prime divisors of $n$ counted  with multiplicity. $\mathbb{P}_{k} = \{ n : \Omega(n) = k\}$ is another example of primitive set (including $\mathbb{P}_0 = \{ 1 \}$). See \cite[Chapter V]{Sequences} for an introduction to the subject of primitive sets.

By $\mathbb{P}$ we denote the set of prime numbers. Let us denote by $p(n)$ and $P(n)$ the minimal and the maximal prime divisors of $n$. 

Erdős in \cite{Erdős_original} proved that for any primitive set $A$
\begin{equation} \label{Erd1}
\sum_{a \in A} \frac{1}{a} \prod_{p \le P(a)} \left(1 - \frac{1}{p} \right) \le 1,
\end{equation}
Here $p$ runs over primes.

From (\ref{Erd1}) Erdős deduced that for any primitive set $A$, $\sum_{a \in A} \frac{1}{a \log a} < M$, where $M$ is an absolute constant.

Indeed:
\[
\sum_{a \in A} \frac{1}{a \log a} \le \sum_{a \in A} \frac{1}{a \log P(a)} \ll \sum_{a \in A} \frac{1}{a} \prod_{p \le P(a)} \left(1 - \frac{1}{p} \right) \le 1.
\]

\begin{defin}
    Let $(\mathbb{P}, \preceq)$ be the set of primes equipped with some linear order. Let $p'(n)$ and $P'(n)$ denote the minimal and the maximal prime divisors of $n$ with respect to the order $\preceq$. Also set $p'(1) = +\infty$. Let $L_a \vcentcolon = \{ a b : b \in \mathbb{N}, P'(a) \preceq p'(b)\}$. A set $A \subset \mathbb{N}$ is called $L$-primitive with respect to $\preceq$ if $a' \notin L_a$ for all distinct $a, a' \in A$. 
\end{defin}
The notion of $L$-primitive sets for the primes ordered by their absolute value was introduced in \cite{Lichtman}. We will call a set $L$-primitive in this sense if the order on primes is not specified. 

Obviously, any primitive set is $L$-primitive with respect to any order, but an $L$-primitive set needs not to be primitive. 

Inspired by the work of Erdős \cite{Erdős_original}, 
we prove a generalization of inequality (\ref{Erd1}):


\begin{theorem} \label{main}
Let $(\mathbb{P}, \preceq)$ be the set of primes with some linear order. Let $A$ be an $L$-primitive set with respect to $\preceq$. Let $f$ be a completely multiplicative function such that $0 \le f(p) \le 1$ for each prime. 
\nopagebreak
Then
\begin{equation} \label{main_ineq}
\sum_{a \in A} f(a) \prod_{p \prec P'(a)} \left(1 - f(p) \right) \le 1.
\end{equation}
\end{theorem}


\subsection{Erdős functions of primitive sets}

Taking $f(p) = z/p$, where $0 < z < 2$, Theorem \ref{main} implies that for each primitive set $A$
\begin{equation} \label{first ineq}
\sum_{a \in A} \frac{z^{\Omega(a)}}{a (\log a)^z} \le \sum_{a \in A} \frac{z^{\Omega(a)}}{a (\log P(a))^z}  \ll \sum_{a \in A} \frac{z^{\Omega(a)}}{a} \prod_{p < P(a)} \left( 1 - \frac{z}{p} \right) \le 1.
\end{equation}

This suggests to introduce a function 
\[
f_z(a) \vcentcolon = \frac{z^{\Omega(a)}}{a (\log a)^z}, \qquad f_z(A) \vcentcolon = \sum_{a \in A} f_z(a),
\]
which is well defined for $z > 0$ and any $A \subset \mathbb{N}$. If the series diverges, then we write $f_z(A) = \infty$.

We call $f_z(A)$ the Erdős function of $A$. 
The sum $f_1(A) = \sum_{a \in A} \frac{1}{a \log a}$ is the classical Erdős sum for $A$.

Let us denote 
\[
\gamma_k(z) \vcentcolon = f_z(\mathbb{P}_k).
\]

The functions $\gamma_k(z)$ can be effectively computed in the similar way as $\gamma_1(1)$ was computed in \cite{Cohen} and $\gamma_k(1)$ were computed in \cite{Lichtman2}. See Proposition \ref{compute_f_k} for details.

In \cite{Lichtman2} it was proved by Lichtman that $\lim_{k \to \infty} \gamma_k(1) = 1$ and that $\gamma_6(1)$ is minimal among $\gamma_k(1)$.
Gorodetsky, Lichtman, Wong in \cite{Gorodetsky} found the second term in asymptotic expansion of $\gamma_k(1)$ as $k$ approaches infinity. 
We improve the result of \cite{Gorodetsky} by finding the third term in asymptotic expansion of $\gamma_k(1)$. 
We also study $\gamma_{k}(z)$ in the range $0 < z \le 2$.

The idea is to study the sum $a^{-1} (\log a)^{-z}$ over numbers of the form $2^{j}b$, where $(b, 2) = 1, \, \Omega(b) = k - j$, using the technique from \cite{Gorodetsky}, and then sum over $j$. 
It turns out that the largest error term comes from $j$  in the neighbourhood of $k - 2\log k$.

\begin{theorem}\label{gamma_k(z) main thm}
Denote
\[
G(z) = \frac{1}{\Gamma(1 + z)} \prod_p \left( 1 - \frac{z}{p} \right)^{-1} \left(1 - \frac{1}{p} \right)^z,
\quad \quad
d_w = 2^{-w} \prod_{p > 2} \left(1 - \frac{w}{p} \right)^{-1} \left( 1 - \frac{1}{p} \right)^w.
\]
By $\gamma$ we denote the Euler-Mascheroni constant. Let $\varepsilon > 0$, then 
$\gamma_{k}(z) = $
\[
\begin{cases}
G(z) + \left( \frac{z}{z + 1} \right)^k \frac{2 d_{z + 1}}{\Gamma(z) (1 - z)}\left(\gamma -\frac{z \log 2 }{1 - z} - \sum_{p > 2} \frac{z \log p}{(p - 1) (p - z - 1)} \right) + O_{\varepsilon} \left( \left( \frac{z}{2} \right)^k k^{2 - z} \right) & (z \in [\varepsilon, 1 - \varepsilon]), \\
 1 - 2^{-k} \frac{d_2}{4} (\log 2) (k^2 - 4k \log k + O(k \sqrt{\log k})) & (z = 1), \\
G(z) + O_{\varepsilon} \left( \left( \frac{z}{2} \right)^k k^{2 - z} \right) & (z \in [1 + \varepsilon, 2 - \varepsilon]), \\
\frac{d_2}{2} \left( k - 2 \log k + O(\sqrt{\log k}) \right) & (z = 2).
\end{cases}
\]
For any $0 < \varepsilon < 1$ uniformly in $z \in (0, 2 - \varepsilon]$ we have 
\[
\gamma_{k}(z) = G(z) + O_{\varepsilon}\left( \left(1 - \frac{\varepsilon}{2} \right)^k k^{\varepsilon} \right).
\]
\end{theorem}

Also note that we have $f_z(2^k) = \frac{z^k}{k 2^k \log 2}$ which tends to infinity with $k$ for $z > 2$.

\begin{rem}
    One can follow \cite{Lichtman2} and prove that for $\varepsilon > 0, \delta > 0$ and uniformly in $0 < z < 2 - \varepsilon$ 
    \[
    \gamma_k(z) = G(z) + O_{\delta, \varepsilon}(k^{-1/2 + \delta}),
    \]
    using  partial summation and Sathe-Selberg theorem (see Lemma \ref{Selb}). 
\end{rem}

As an application of Theorem \ref{gamma_k(z) main thm} we prove the following
\begin{prop} \label{H-L Tauber}
    Suppose that $\sum_{1 \le k \le x} h(k) \sim c x$, where $h(k) \in \mathbb{C}$. And suppose that for each $z > 1$ the sum $\sum_{n = 2}^{\infty} \frac{h(\Omega(n))}{n (\log n)^z}$ is absolutely convergent (for example this is the case if $h(k) \in \mathbb{R}_{\ge 0}$ for all $k$). Then
    \[
    \lim_{z \to 1+} (z - 1) \sum_{n = 2}^{\infty} \frac{h(\Omega(n))}{n (\log n)^z} = c.
    \]
\end{prop}

In 1986, Erdős \cite[Conjecture 2.1]{Erd_conj} asked if the maximum of $\sum_{a \in A} \frac{1}{a \log a}$ among all primitive sets is attained by $A = \mathbb{P}$. In 2022 this Conjecture was proved by Lichtman \cite{Lichtman} (see \cite{Lichtman} for the history of progress towards the proof of this Conjecture).

In the same manner we are interested in
\[
U(z) \vcentcolon = \sup_{A \,\, \text{primitive}} f_z(A).
\]

The basic properties of $U(z)$ are
\begin{theorem} \label{g behaviour}
\,

    A) $U(z)$ is bounded on $(0, 2 - \varepsilon]$ for each $\varepsilon > 0$,

    B) $U(z) = \infty$ for $z \ge 2$,

    C) $\lim_{z \to 2-} U(z) = \infty$.
\end{theorem}

\begin{proof}
A) follows from (\ref{first ineq}), since it is uniform on $[0, 2 - \varepsilon]$.
B) and C) follow from Theorem \ref{gamma_k(z) main thm} since $\lim_{k \to \infty} \gamma_{k}(z) \le U(z)$ and $G(z) = \lim_{k \to \infty} \gamma_{k}(z)$ has a pole at $2$.
\end{proof}

\begin{defin}
    Let $z > 0$. We say that a prime $p$ is Erdős $z$--strong if for any primitive $A \subset L_{p}$ we have $f_{z}(p) \ge f_z(A)$.
\end{defin}

This notion was introduced in \cite{Pomerance} for $z = 1$  (such primes are called Erdős strong). In \cite{Lichtman} it was proved that all odd primes are Erdős $1$--strong. It remains an open question if $2$ is Erdős $1$--strong. Obviously if all primes are Erdős $z$--strong, then $U(z) = \gamma_1(z)$, which we know is not the case for $z$ that is close to $2$ since then $U(z) \ge G(z) > \gamma_1(z)$.

We follow \cite{Lichtman} and prove that $U(z) = \gamma_1(z)$ in some neighborhood of $z = 1$. 

\begin{theorem} \label{U(z) main}
    A) All odd primes are Erdős $z$--strong for $z \in [0.44, 2]$.

    B) $U(z) = \gamma_1(z)$ for $z$ in some neighborhood of $1$. 
    
    C) For any $\varepsilon > 0$ there exists $N$ such that all primes greater than $N$ are Erdős $z$--strong for any $z \in [\varepsilon, 2]$.
\end{theorem}

The sum of $1/(n\log n)$ over a primitive set is convergent. For any $\varepsilon > 0$ the sum of $1 / (n (\log n)^{\varepsilon})$ over $\mathbb{P}_k$ converges for any $k \ge 1$ (but in view of Theorem \ref{gamma_k(z) main thm} it is not uniformly bounded in $k$).  We can ask whether we can replace $1/(n\log n)$ with some function which decays slower, but with the property, that the sum of this function over an arbitrary primitive set is convergent. The answer is negative:

\begin{theorem} \label{tightness}
For each function $\psi(n)$ such that $\lim_{\Omega(n) \to \infty} \psi(n) = +\infty$ there exists a primitive set $A$, which satisfies two properties:

I. For any $0 < z < 2$
\[
\sum_{a \in A} \frac{\psi(a) z^{\Omega(a)}}{a (\log a)^z} = \infty,
\]

II. $f_z(A) = \infty$ for any $z \ge 2$. 

\end{theorem}

\subsection{Primitive density}

Now we know, that $f_z(A)$ converges for any $0 < z < 2$ and can diverge for $z \ge 2$.

We want to define $f_0(A)$ somehow.

\begin{defin}
Let $A$ be an arbitrary set of natural numbers, and assume that $1 \notin A$.
\[
\overline{\eta}(A) \vcentcolon = \limsup_{z \to 0} f_z(A), 
\quad \underline{\eta}(A) \vcentcolon = \liminf_{z \to 0} f_z(A).
\]
We call $\overline{\eta}(A)$ and $\underline{\eta}(A)$ the upper and the lower primitive densities of $A$ respectively.

If $\overline{\eta}(A) = \underline{\eta}(A)$, then we say that $A$ has a primitive density $\eta(A) \vcentcolon = \overline{\eta}(A) = \underline{\eta}(A)$.
\end{defin}

It seems natural to define $\eta(\{ 1 \}) = 1$.

Note that this notion is not a usual density in a sense that $\eta(\mathbb{N}) = \infty \ne 1$.

\begin{theorem} \, \label{pr_dens}

    A) For each primitive set $A$, \, $0 \le \underline{\eta}(A) \le \overline{\eta}(A) \le 1$;

    B) $\forall k \ge 0, \, \, \, \eta(\mathbb{P}_k) = 1$;

    C) There exists a primitive set $A$ such that $\underline{\eta}(A) = 0$ and $\overline{\eta}(A) = 1$;

    D) Let $A \subset \mathbb{P}_k$. Let us say that $A$ has Dirichlet density $c$ if 
    \[
    \lim_{s \to 1+} \frac{\sum_{a \in A} a^{-s}}{\sum_{m \in \mathbb{P}_k} m^{-s}} = c
    \]
    (see \cite[Chapter VI, \S 4]{serre2012course} for the definition and discussion of Dirichlet density). Suppose that $A$ has Dirichlet density $c$, then it has primitive density $c$;

    E) If $\overline{\eta}(B) > n - 1, \,\, n \in \mathbb{N}$, then there exists an infinite matrix \, $(b_{j, i})_{j \in \mathbb{N}, 1 \le i \le n}$ such that $b_{j, i} \in B$, all $b_{j, i}$ are distinct and for all $j \in \mathbb{N}$ and $1 \le i \le n - 1$ we have $b_{j, i} | b_{j, i+1}$;

    F) There exists a set $B$ such that $\eta(B) = \infty$, but $B$ does not contain an infinite sequence $b_1, b_2, \ldots$ such that $b_{i} | b_{i+1}$ for all $i$.

\end{theorem}

Theorem \ref{pr_dens}F is a negative result in some sense. Since primitive sets have logarithmic density zero it is easy to see that if the upper logarithmic density of $B$ is positive, then for arbitrary large $k$, $B$ contains a subsequence $b_1, b_2, \ldots, b_k$ such that $b_{i} | b_{i + 1}$. The Davenport-Erdős theorem \cite{Davenport-Erdos} states that actually $B$ contains an infinite subsequence $b_1, b_2, \ldots$ such that $b_{i} | b_{i+1}$. 
Theorem \ref{pr_dens}F shows that the upper primitive density does not possess this property.

\subsection{$z$-logarithmic density of primitive sets}

Other questions arise when we study primitive subsets of $\{1, 2, \ldots, N \}$. 

For $A \subset \{1, 2, \ldots, N \}$ we define 
\[
\delta(f, A, N) =  \frac{\sum_{a \in A} f(a)}{\sum_{n \le N} f(n)}.
\]

Let us define $z$-logarithmic density of $A \subset \{1 , 2, \ldots, N \}$ in $\{1, 2, \ldots, N\}$ to be $\delta(h_z, A, N)$, where $h_z(n) \vcentcolon = z^{\Omega(n)} / n$. By $D_z(N)$ we denote the maximal $z$-logarithmic density other all primitive subsets of $\{1, \ldots, N \}$:
\[
D_z(N) = \max_{A \,\, \text{primitive}} \delta(h_z, A, N).
\]

If $z < 2$, then Theorem \ref{main} implies that $D_z(N) = o(1)$.

For the classical case $z = 1$ it was proved by Behrend in \cite{Behrend}, that $D_1(N) \ll (\log \log N)^{-1/2}$. Pillai showed in \cite{Pillai} that actually $D_1(N) \gg (\log \log N)^{-1/2}$. 

In \cite{Szemeredi} it was proved by Erdős, Sarkozi and  Szemeredi, that $D_1(N) \sim (2 \pi \log \log N)^{-1/2}$.

\begin{theorem} \label{D_z}
    A) For $0 < z < 2$ we have as $N$ tends to infinity
    \begin{flalign*}
    && D_z(N) &\sim (2 \pi z \log \log N)^{-1/2};&\\
    \quad\,\,\,\text{B)} && D_2(N) &\asymp \left( \log N \right)^{-1};&
   \end{flalign*}

    C) If $z > 2$, then there exists $C > 0$, $N_0$ such that for all $N > N_0$ we have $D_z(N) \ge C$. One can take $N_0 = 4$ and $C = (1 - 2/z)/3$.

\end{theorem}

Upper bound in Theorem \ref{D_z}B follows from 

\begin{prop} \label{improvement}
Let $A$ be a primitive set. Then
    \[
    \sum_{
    \substack{a \in A \\ P(a) \le N}} \frac{2^{\Omega(a)}}{a} \le \prod_{2 < p \le N} (1 - 2/p)^{-1} \ll (\log N)^2.
    \]
\end{prop} 

Note that Theorem \ref{D_z}B implies that the left hand side in this inequality is $\gg (\log N)^2$ for some primitive set $A$.

Theorem \ref{main} is applied to prove Proposition \ref{improvement}. Proposition \ref{improvement} improves Lemma 2 in \cite{Elem_upperbound}, that states that
\[
\sum_{\substack{a \in \mathbb{P}_k \\ P(a) \le N}} \frac{1}{n} \ll (k+1) 2^{-k}(\log N)^2.
\]

The universal upper bound for $N_k(x) \vcentcolon = |\mathbb{P}_k \cap [1, x]|$ is given in \cite{Elem_upperbound}. It is proved, that $N_k(x) \ll k^4 2^{-k} x \log x$. 

Using methods of complex analysis it was proved that

\begin{prop}[Balazard-Delange-Nicolas] \label{BDN}
For $k \ge 1$ and $x/2^k \to \infty$
\[
N_k(x) \sim (2 - \rho) G(\rho) \frac{x}{2^k} \left(\log \frac{x}{2^k} \right)^{-1} \sum_{0 \le j < k} \frac{\left( 2 \log \log \frac{x}{2^k} \right)^j}{j!},
\]
where
\[
\rho \sim \min\left\{ 2, \frac{k-1}{\log \log \frac{x}{2^k}}\right\}, \quad \rho < 2.
\]
\end{prop}

\begin{proof}
    See \cite{Balazard} or \cite{Hwang}. See also \cite[Chapter II.6 and Notes]{Tenenbaum} for the related discussion. 
\end{proof}

This gives an immediate 

\begin{cor} For $x \ge 2$ and $k \ge 1$ 
\[
N_k(x) \ll \frac{x}{2^k} \left(\log \frac{x}{2^k} \right) + 1.
\]
\end{cor}

Using Proposition \ref{improvement} instead of \cite[Lemma 2]{Elem_upperbound} we give an elementary proof of much weaker result:

\begin{theorem} \label{N_k-constant}
For all $x \ge 3$ and $k \ge 1$
    \begin{equation} \label{le ?}
    N_k(x) \le 1.35 \, k^3 \frac{x \log x}{2^k}.
    \end{equation}
\end{theorem}

The proof remains the same as in \cite{Elem_upperbound} and we just do some numerical estimates for the constant.

\subsection{Open questions}
Is $U(z)$ continuous on $(0, 2)$? For which $z$ do we have $U(z) = \gamma_1(z)$? Is this the case for $z$ that are sufficiently close to $0$?
What is the asymptotic of $D_2(N)$?

\begin{hypo*}
$U(z) = \gamma_1(z)$ for $z \in (0, 1]$.
\end{hypo*}

\section{Erdős inequality for primitive sets: Theorem \ref{main}}

\begin{lemma} \label{L_a intersection}
    Let the set of primes be equipped with some linear order. For any integers $a, a'$, if $L_a \cap L_{a'} \ne \varnothing$, then $a \in L_{a'}$ or $a' \in L_a$. 
\end{lemma}

\begin{proof}
For $n \in \mathbb{N}$ we define a sequence of numbers by the recursive formula $n_1 = p'(n)$, $n_k = p'(n / n_{k-1}) n_{k-1}$. Then $n_r = n$ for $r \ge \Omega(n)$. 

Obviously 
$n_k | n_{k + 1}$ for each $k$. 
Note that $n \in L_a$ iff $a = n_r$ for some $r \ge 1$. Moreover $n_{l} \in L_{n_k}$ for each $l \ge k$.
Hence if $n \in L_a \cap L_{a'}$, then $a = n_l, a' = n_k$ for some $l, k$. If $l \ge k$, then $a \in L_{a'}$ and if $l < k$, then $a' \in L_a$.
\end{proof}

\begin{proof}[Proof of Theorem \ref{main}]

It is enough to prove inequality (\ref{main_ineq}) for finite $A$, because all terms in the sum over $A$ are non-negative. 

Denote by $P$ the set of prime divisors of $(\prod_{a \in A} a)$. Let us introduce a new linear order $\preceq'$ on $\mathbb{P}$, such that $\preceq'$ and $\preceq$ coincide on $P$ and for any $p_1 \in P, p_2 \in \mathbb{P} \setminus P$ we let $p_1 \preceq' p_2$. 

Then
\[
\sum_{a \in A} f(a) \prod_{p \prec P'(a)} \left(1 - f(p) \right) \le 
\sum_{a \in A} f(a) \prod_{\substack{p \prec P'(a) \\ p \in P}} \left(1 - f(p) \right) = \sum_{a \in A} f(a) \prod_{p \prec' P'(a)} \left(1 - f(p) \right).
\]

Hence it is enough to prove inequality (\ref{main_ineq}) for linear orders such that for any $p_1 \in P, p_2 \in \mathbb{P} \setminus P$ we have $p_1 \preceq p_2$. We assume that $\preceq$ has such property. In particular for any $a \in A$ the set $\{p : p \prec P'(a)\}$ is finite.

Also it is enough to prove (\ref{main_ineq}) in the case $0 < f(p) < 1$ for each $p$, because then we can take a limit.

For $a \in A$ let us denote $P_a \vcentcolon = \prod_{p \prec P'(a)} p$. 
Let $g$ be a non-negative completely multiplicative arithmetic function and suppose that $\sum_{n} g(n)$ converges to a positive number. The sets $L_a$ are disjoint for distinct $a \in A$, otherwise Lemma \ref{L_a intersection} gives a contradiction with $L$-primitiveness of $A$. Hence
\[
\sum_{\substack{n \in \mathbb{N}}} g(n) \ge \sum_{a \in A} \sum_{n \in L_a} g(n) = \sum_{a \in A} \sum_{\substack{P'(a) \preceq p'(b)}} g(a b) =  \sum_{a \in A} g(a)  \sum_{\substack{P'(a) \preceq p'(b)}} g(b) =
\]
\[
\sum_{a \in A} g(a)  \sum_{(b, P_a) = 1} g(b) = \sum_{a \in A} g(a) \sum_{m \in \mathbb{N}} \sum_{d | P_a} \mu(d) g(d m) = 
\]
\[
\sum_{m \in \mathbb{N}} g(m) \sum_{a \in A} g(a) \sum_{d | P_a} \mu(d) g(d) = \left( \sum_{m \in \mathbb{N}} g(m) \right) \sum_{a \in A} g(a) \prod_{p \prec P'(a)} \left(1 - g(p) \right).
\]
Now we divide both sides of this inequality by $\sum_n g(n)$. This gives us
\[
\sum_{a \in A} g(a) \prod_{p \prec P'(a)} \left(1 - g(p) \right) \le 1.
\]

Now take $g(p) = f(p)$ for $p | P$ and $g(p) = 0$ for other primes. The sum $\sum_n g(n)$ converges and we obtain (\ref{main_ineq}).
\end{proof}

\section{$\gamma_k(z)$ functions: Theorems \ref{gamma_k(z) main thm}, \ref{tightness} and Proposition \ref{H-L Tauber}} \label{computations}

\begin{lemma} \label{sum to int}
\[
\frac{1}{n (\log n)^{z}} = \frac{1}{\Gamma(z)} \int_1^{\infty} n^{-s} (s - 1)^{z - 1} ds.
\]
\end{lemma}

\begin{proof}
\[
\int_1^{\infty} n^{-s} (s - 1)^{z - 1} ds = \frac{1}{n} \int_0^{\infty} e^{- s \log n} s^{z - 1} ds = 
\]
\[
\frac{1}{n (\log n)^z} \int_0^{\infty} e^{-t} t^{z-1} dt = \frac{\Gamma(z)}{n (\log n)^z}.
\]
\end{proof}

\begin{prop} \label{compute_f_k}
Let $P_k(s) \vcentcolon = \sum_{\Omega(n) = k} n^{-s}$, $P(s) = P_1(s)$. Then
\[
P(s) = \sum_{m \ge 1} \frac{\mu(m)}{m} \log \zeta(ms), \quad  P_k(s) = \frac{1}{k} \sum_{j = 1}^{k} P(j s) P_{k - j}(s),
\]
\[
\gamma_k(z) = \frac{z^k}{\Gamma(z)} \int_1^{\infty} P_k(s) (s - 1)^{z - 1} ds.
\]
\end{prop}
\begin{proof}
    The last formula follows from Lemma \ref{sum to int}. See \cite[Section 3]{Lichtman2} for the first two formulas. 
\end{proof}

This proposition allows us to compute $\gamma_k(z)$ efficiently.

Following \cite{Gorodetsky}
let us introduce for $s \ge 1$ and $|w| < 2$
\[
F_{y}(s, w) \vcentcolon = \sum_{\substack{n \ge 1 \\ p(n) > y}} \frac{w^{\Omega(n)}}{n^s},
\quad \quad
G_y(s, w) \vcentcolon = F_{y}(s, w) (s-1)^{w} = F_{y}(s, w) \zeta(s)^{-w} (\zeta(s) (s-1))^{w}.
\]

The function $G(s, w)$ for each fixed $s$ has a meromorphic continuation to the whole complex plane and it has simple poles at $w = p^{s}$.

In particular
\begin{equation} \label{G_y(1, w)}
G_y(1, w) = \prod_{p \le y} \left(1 - \frac{1}{p} \right)^{w} \prod_{p > y} \left(1 - \frac{w}{p} \right)^{-1} \left( 1 - \frac{1}{p} \right)^w, 
\end{equation}

For a smooth function $H(s, w)$ by $[w^i]H(s, w)$ we denote the coefficient of $w^i$ in the Taylor expansion of $H(s, w)$ at $w = 0$, which is a function of $s$.
Also denote
\[
H^{(a, b)}(s, w) \vcentcolon = \frac{\partial^{a + b}}{\partial s^a \partial w^b} H(s, w).
\]

\begin{lemma} \label{G_y bound}
Let $y_1$ be the smallest prime greater than $y$. For $y \ge 2$ and $m \ge 0$,
\[
[w^i] G_y^{(m, 0)}(s, w) \ll_{m, y} \frac{(i + 1)^m}{y_1^i}
\]
uniformly for $s \in [1, 2]$ and $i \ge 0$.
\end{lemma}

\begin{proof}
    See \cite[Remark 2.4]{Gorodetsky}.
\end{proof}

We have $\gamma_k(z) = \sum_{0 \le j \le k} \gamma_{k, j}(z)$, where
\[
\gamma_{k, j}(z) = \sum_{\substack{\Omega(n) = k \\ 2^{j} || n}} \frac{z^k}{n (\log n)^z}.
\]
Let us evaluate $\gamma_{k, j}(z)$.
Denote 
\[
P_{k, j}(s) \vcentcolon = \sum_{\substack{\Omega(n) = k \\ 2^{j} || n}} \frac{1}{n^s} = \sum_{\substack{\Omega(n) = k - j \\ (n , 2) = 1}} \frac{2^{-js}}{n^s}.
\]
Assume that $z \le 2$. Lemma \ref{sum to int} implies
\[
\gamma_{k, j}(z) = \frac{z^k}{\Gamma(z)} \int_1^{\infty} P_{k, j}(s) (s - 1)^{z-1} ds = \frac{z^k}{\Gamma(z)} I_{k, j} + B_{k, j},
\]
where
\[
I_{k, j} = \int_1^{2} P_{k, j}(s) (s - 1)^{z-1} ds.
\]
We have
\begin{multline*}
B_{k, j} = \frac{z^k}{\Gamma(z)} \int_2^{\infty} P_{k, j}(s) (s - 1)^{z-1} ds 
\ll z^{k + 1} \sum_{\substack{\Omega(n) = k \\ 2^{j} || n}} \frac{1}{n (\log n)^z} \int_{\log n}^{\infty} e^{-t} t^{z-1} dt = \\
z^{k + 1} \sum_{\substack{\Omega(n) = k \\ 2^{j} || n}} \frac{1}{n (\log n)^z} \left( \frac{(\log n)^{z-1}}{n} + (z-1) \int_{\log n}^{\infty} \frac{e^{-t} t^{z-1}}{t} dt \right)  \ll \\
z^{k+1} \sum_{\substack{\Omega(n) = k \\ 2^{j} || n}} \frac{1}{n^2 (\log n)} \ll 
z^{k+1} 2^{-2j} \sum_{m \ge 3^{k - j}} \frac{1}{m^2 (\log 2^j m)} \ll 
\frac{z^{k + 1} 3^{-k} (3/4)^j}{k}.
\end{multline*}

And
\[
I_{k, j} = \int_1^2 (s-1)^{z-1} 2^{-js} [w^{k-j}] F_2(s, w) \, ds.
\]

We have
\[
F_2(s, w) = (s-1)^{-w} G_2(s, w).
\]

Hence
\[
I_{k, j} = \sum_{l + i = k - j} \int_1^2 2^{-js} (s-1)^{z-1} \frac{(-\log(s-1))^l}{l!} \frac{1}{i!} G_2^{(0, i)}(s, 0) \, ds.
\]

Now we introduce
\[
I_{k, j}' = \sum_{l + i = k - j} \int_1^2 2^{-js} (s-1)^{z-1} \frac{(-\log(s-1))^l}{l!}  \frac{1}{i!} G_2^{(0, i)}(1, 0) \, ds.
\]

\[
I_{k, j}'' = \sum_{l + i = k - j} \int_1^2 2^{-js} (s-1)^{z} \frac{(-\log(s-1))^l}{l!}  \frac{1}{i!} G_2^{(1, i)}(1, 0) \, ds.
\]

\begin{lemma} \label{I_k = I_k'}
   For $0 < z \le 2$ we have 
   \[
   |I_{k, j} - I_{k, j}'| \ll  
    2^{-k} \left( \frac{2}{1 + z} \right)^{k - j} \sum_{i \le k - j} \left( \frac{1 + z}{3} \right)^i (i + 1).
   \]
    \[
    |I_{k, j} - I_{k, j}' - I_{k, j}''| \ll  
    2^{-k} \left( \frac{2}{2+z} \right)^{k - j} \sum_{i \le k - j} \left( \frac{2 + z}{3} \right)^{i} (i + 1)^2.
    \]
\end{lemma}

\begin{proof}[Proof of Lemma \ref{I_k = I_k'}]
    Lemma \ref{G_y bound} implies
    \[
    |G_2^{(0, i)}(s, 0) - G_2^{(0, i)}(1, 0)| \le (s - 1) \sup_{s' \in [1, 2]} |G_2^{(1, i)}(s, 0)| \ll (s-1) \frac{(i + 1)!}{3^i}.
    \]
    Thus
    \begin{multline*}
    |I_{k, j} - I_{k, j}'| \ll \sum_{l + i = k - j} \int_1^2 2^{-js} (s-1)^z  \frac{(-\log(s-1))^l}{l!} (i+1) 3^{-i} \, ds = \\
    \sum_{l + i = k - j} \int_0^{\infty} \frac{t^l}{l!} 2^{-j(1 + e^{-t})} e^{-t - zt}  (i+1) 3^{-i} \, dt \le \\
    \sum_{l + i = k - j} (i+1) 3^{-i} 2^{-j} \int_0^{\infty} \frac{t^l}{l!}  e^{-t (1 + z)}  \, dt = 
    \sum_{l + i = k - j} (i+1) 3^{-i} 2^{-j} (1 + z)^{-l-1}.
    \end{multline*}

    For $z \in (0, 2]$ we obtain
    \[
    |I_{k, j} - I_{k, j}'| \ll 2^{-j} \sum_{l + i = k - j} 3^{-i} (1 + z)^{-l} (i + 1)
    =
    2^{-k} \left( \frac{2}{1 + z} \right)^{k - j} \sum_{i \le k - j} \left( \frac{1 + z}{3} \right)^i (i+1).
    \]

    Lemma \ref{G_y bound} implies
    \[
    |G_2^{(0, i)}(s, 0) - G_2^{(0, i)}(1, 0) - (s-1) G_2^{(1, i)}(1, 0)| \le \frac{(s - 1)^2}{2} \sup_{s' \in [1, 2]} |G_2^{(2, i)}(s, 0)| \ll (s-1)^2 \frac{(i + 2)!}{3^i}.
    \]
    In the same way we obtain
    \[
    |I_{k, j} - I_{k, j}' - I_{k, j}''| \ll  
    2^{-k} \left( \frac{2}{2+z} \right)^{k - j} \sum_{i \le k - j} \left( \frac{2 + z}{3} \right)^{i} (i + 1)^2.
    \]
\end{proof}

Now we need to evaluate $I_{k, j}'$. Let us take $0 < \epsilon < z$. Then by Cauchy’s integral formula
\[
I_{k, j}' = \int_1^2 2^{-js} [w^{k - j}] \left( (s - 1)^{z-w-1} G_2(1, w) \right) ds = 
\]
\[
\frac{1}{2\pi i}\int_1^2 \oint_{|w| = \epsilon} 2^{-js} \frac{(s - 1)^{z-w-1} G_2(1, w)}{w^{k - j +1}} \, dw \, ds.
\]

By Fubini's theorem
\[
I_{k, j}' = \frac{1}{2\pi i} \oint_{|w| = \epsilon} \frac{G_2(1, w)}{w^{k - j + 1}}
\left( \int_1^2 2^{-js} (s - 1)^{z-w-1} ds \right) \, dw.
\]

Denote
\[
h_{j, z}(w) \vcentcolon = \int_1^2 2^{-js} (s - 1)^{z-w-1} ds.
\]
Integrating by parts, we obtain 
\[
h_{j, z}(w) = 2^{-2j} \left( \frac{1}{z-w} + \frac{j \log 2}{(z-w)(z-w+1)} + \frac{(j \log 2)^2}{(z-w)(z-w + 1) (z - w + 2)} + \ldots \right).
\]
This gives a meromorphic continuation of $h_{j, z}(w)$ to the whole complex plane. 

This series converges uniformly in $w$ on compact subset of $\mathbb{C}$ which do not contain its poles.

Take $R > \epsilon$, $R \notin \mathbb{Z} \cup (z +\mathbb{Z})$. Then $I_{k, j}' = R_{k, j} + M_{k, j}$, where
\[
R_{k, j} \vcentcolon = 
\frac{1}{2\pi i} \oint_{|w| = R} \frac{G_2(1, w)}{w^{k - j +1}}
h_{j, z}(w) \, dw, \quad \quad
M_{k, j} \vcentcolon = 
- \sum_{\epsilon < |w| < R} \operatorname{Res}_{w} \frac{G_2(1, w)}{w^{k - j +1}} h_{j, z}(w). 
\]

To evaluate the error term $R_{k, j}$ we will need a few lemmas.

\begin{lemma} \label{a^delta}
Let $a \ge 1/100$, $0 \le \delta \le 100$. Then
\[
\sum_{m \ge 0} \frac{a^m m^{\delta}}{m!} \ll a^{\delta} e^a.
\]
\end{lemma}
\begin{proof}

For all $m \ge [200 a] + 1$ we have
\[
\left( \frac{a^{m + 1} (m + 1)^{\delta}}{(m+1)!} \right) \left( \frac{a^{m} m^{\delta}}{m!} \right)^{-1} \le \frac{1}{2}.
\]
Hence
\[
\sum_{m \ge [200 a] + 1} \frac{a^m m^{\delta}}{m!} \ll \frac{a^{[200 a] + 1} a^{\delta}}{([200 a] + 1)!} \ll a^{\delta} e^{a}.
\]
And
\[
\sum_{m \le [200 a] + 1} \frac{a^m m^{\delta}}{m!} \ll
([200 a] + 1)^{\delta} \sum_{m \le [200 a] + 1} \frac{a^m}{m!} \ll a^{\delta} e^{a}.
\]
\end{proof}

\begin{lemma} \label{j^Re}
    Let $0 < z \le 2, j \ge 1, w \in \mathbb{C}, |w| \le 10$ and let $q$ be the closest number to $w$ of the form $z + m$, where $m \in \mathbb{Z}_{\ge 0}$. If $w \ne q$, then
    \[
    h_{j, z}(w) \ll 2^{-j}  \, \, \frac{j^{\operatorname{Re}(w) - z}}{|w - q|}.
    \]
\end{lemma}

\begin{proof}
We have
\begin{multline*}
h_{j, z}(w) = \frac{2^{-2j}}{z-w} \left(1 + \sum_{m = 1}^{\infty} \frac{(j \log 2)^m}{m!} \prod_{l = 1}^m \left( \frac{z - w + l}{l} \right)^{-1} \right) \ll \\
2^{-2j} \left( \prod_{r = 0}^{15} (z - w + r)^{-1} \right) \sum_{m = 0}^{\infty}  \frac{(j \log 2)^m}{m!} \prod_{15 < r \le m} \left( \frac{z - w + l}{l} \right)^{-1} .
\end{multline*}

And for $m \ge 1$
\begin{multline*}
\prod_{15 < r \le m} \left( \frac{z - w + l}{l} \right)^{-1} = \exp \left(\sum_{15 < r \le m} - \log \left( 1 - \frac{w - z}{l} \right)  \right) \ll \\
\exp \left( \operatorname{Re}(w - z) \sum_{15 < r \le m} \frac{1}{r} \right) \ll \exp \left(  \operatorname{Re}(w - z) (\log m) \right) = 
m^{\operatorname{Re} w - z}.
\end{multline*}

Hence 
\[
h_{j, z}(w) \ll \frac{2^{-2 j}}{|w - q|} \sum_{m = 0}^{\infty} \frac{(j \log 2)^m \, m^{\operatorname{Re} w - z}}{m!} \ll \frac{2^{-j}}{|w - q|} j^{\operatorname{Re} w - z}.
\]
The last inequality follows from Lemma \ref{a^delta}.
\end{proof}

\begin{lemma} \label{int bound j^R}
    Let $1/5 < R < 5, j \ge 10^3$ and let $q$ be the closest number to $R$ of the form $z + m$, where $m \in \mathbb{Z}_{\ge 0}$. Let $\delta \vcentcolon = |q - R|$. Then
    
    \begin{align*} 
    \int_{-\pi}^{\pi} \frac{j^{\operatorname{Re}(R e^{i \varphi})}}{|q - R e^{i \varphi}|} \, d \varphi &\ll 
    \frac{j^R}{\delta \sqrt{\log j}}, 
    \quad \quad &\text{if} \quad \delta \gg (\log j)^{-1/2}; \\
    \int_{-\pi}^{\pi} \frac{j^{\operatorname{Re}(R e^{i \varphi})}}{|q - R e^{i \varphi}|} \, d \varphi &\ll 
    j^R \log \left( \frac{1}{\delta \sqrt{\log j}} \right), 
    \quad \quad &\text{if} \quad \delta \ll (\log j)^{-1/2}.
    \end{align*}
\end{lemma}

\begin{proof}
Denote $b \vcentcolon = q / R$, $B = j^R > 3$.
\[
\int_{-\pi}^{\pi} \frac{j^{\operatorname{Re}(R e^{i \varphi})}}{|q - R e^{i \varphi}|} \, d \varphi = 
R^{-1} \int_{-\pi}^{\pi} \frac{j^{R \cos \varphi}}{|b - e^{i \varphi}|} \, d \varphi.
\]

Suppose that $\delta \gg 1$, then
\[
\int_{-\pi}^{\pi} \frac{j^{R \cos \varphi}}{|b - e^{i \varphi}|} \, d \varphi \ll \int_{0}^{\pi} B^{\cos \varphi} \, d \varphi \ll 
\int_{0}^{\pi/2} B^{\cos \varphi} \, d\varphi.
\]
Note that $\cos \varphi \ge 1 - \varphi^2 / 4$ on $[-\pi/2, \pi/2]$. Hence
\[
\int_{0}^{\pi/2} B^{\cos \varphi} \, d\varphi \ll B \int_{0}^{\pi / 2} B^{- \varphi^2 / 4} \, d \varphi \ll \frac{B}{\sqrt{\log B}} \int_{0}^{\pi \sqrt{\log B}/2} e^{-t^2 / 4} dt \ll \frac{B}{\sqrt{\log B}}.
\]

This proves the Lemma in case $\delta \gg 1$. Now suppose that $\delta < 1/100$.

We have $|b - e^{i \varphi}| \ge \max(\delta, |\sin \varphi|) \gg \max(\delta, |\varphi|)$. Thus
\[
\int_{-\pi}^{\pi} \frac{j^{R \cos \varphi}}{|b - e^{i \varphi}|} \, d \varphi \ll \frac{1}{\delta} \int_0^{\delta} B^{\cos \varphi} d \varphi + \int_{\delta}^{\pi/2} \frac{B^{\cos \varphi}}{\varphi} \, d \varphi.
\]

We have
\[
\frac{1}{\delta} \int_0^{\delta} B^{\cos \varphi} d \varphi \ll
\frac{B}{\delta} \int_{0}^{\delta} B^{-\varphi^2 / 4} d \varphi \ll \frac{B}{\delta \sqrt{\log B}} \int_0^{\delta \sqrt{\log B}} e^{-t^2 / 4} \, dt \ll
\frac{B \min(1, \delta \sqrt{\log B})}{\delta \sqrt{\log B}}.
\]
\[
\int_{\delta}^{\pi/2} \frac{B^{\cos \varphi}}{\varphi} \, d \varphi \ll 
B \int_{\delta}^{\pi/2} \frac{B^{-\varphi^2 / 4}}{\varphi} \, d \varphi \ll B \int_{\delta \sqrt{\log B}}^{\pi \sqrt{\log B}/2} \frac{e^{-t^2/4}}{t} \, dt.
\]
If $\delta \gg (\log j)^{-1/2}$, then
\[
B \int_{\delta \sqrt{\log B}}^{\pi \sqrt{\log B}/2} \frac{e^{-t^2/4}}{t} \, dt \ll 
B \frac{e^{-(\delta \sqrt{\log B})^2}}{\delta \sqrt{\log B}} \ll \frac{B}{\delta \sqrt{\log B}}.
\]
and if $\delta \ll (\log j)^{-1/2}$, then
\[
B \int_{\delta \sqrt{\log B}}^{\pi \sqrt{\log B}/2} \frac{e^{-t^2/4}}{t} \, dt \ll 
B \left(1 + \int_{\delta \sqrt{\log B}}^1 \frac{dt}{t}\right) \ll
B \log \left( \frac{1}{\delta \sqrt{\log B}} \right).
\]

Putting all things together we obtain the result.
\end{proof}

\begin{proof}[Proof of Theorem \ref{gamma_k(z) main thm}]

Denote $a \vcentcolon = (k - j) / \log k$.

Suppose that $j \le k - 2.5 \log k$. We take $R \in [2.5, 2.501]$, in such way that $\delta = \min(|R - z|, |R - z + 1|, |R - z + 2|)  \ge 10^{-4}$. Then
\[
R_{k, j} \ll 2^{-j} k^{2.501 - z} (2/5)^{k - j} \ll 2^{-k} k^{2.501 - z + a (\log 2 - \log(5/2))}.
\]
This implies
\[
\sum_{0 \le j \le k - 2.5 \log k} |R_{k, j}| \ll 
2^{-k} (\log k) k^{2.501 + 2.5 (\log 2 - \log(5/2)) - z} \ll 2^{-k} k^{2 - z}.
\]

If $a \le 1 - \frac{1}{\sqrt{\log k}}$, let us take $R \in [1, 1 - (\log k)^{-1/2}]$, so that $\delta \gg (\log k)^{-1/2}$. Lemma \ref{int bound j^R} gives us
\[
R_{k, j} \ll 2^{-j} k^{R - z} R^{j - k} \ll
2^{-k} k^{R - z + a (\log 2 - \log R)} \ll
2^{-k} k^{1 + \log 2 - z}.
\]
This implies
\[
\sum_{k - \log k \le j \le k} |R_{k, j}| \ll 2^{-k} k^{2 - z}.
\]

In other cases let $q$ be the closest number to $a$ of the form $z, z + 1$ or $z + 2$.

If $|a - q| \ge (\log k)^{-1/2}$, then we take $R = a$. Otherwise, if $a - q \ge 0$, let us take $R = a + (\log k)^{-1/2}$ and if $a - q < 0$ we take $R = a - (\log k)^{-1/2}$. In all cases $\delta \sqrt{\log k} \gg 1$.

Denote $h(a) \vcentcolon = a + a(\log 2 - \log a)$. 
Using Lemma \ref{int bound j^R} we obtain
\[
R_{k, j} \ll 2^{-k} \frac{k^{R - z + a (\log 2 - \log R)} }{\max(1, \delta \sqrt{\log k})} \ll 2^{-k} \frac{k^{h(a) - z}}{\max(1, \delta \sqrt{\log k})}.
\]

We have
\[
\gamma_k(z) = \frac{z^k}{\Gamma(z)} \sum_{j = 0}^{k} M_{k, j} + O\left(z^{k + 1} \sum_{j = 0}^{k} (|R_{k, j}| + |I_{k, j} - I_{k, j}'|) \right).
\]

If $z \in [\varepsilon, 1 - \varepsilon] \cup [1 + \varepsilon, 2 - \varepsilon]$, then
\[
z^{k+1} \sum_{a : |a - q| < \varepsilon/2} |R_{k, j}| \ll (z/2)^{k+1} (\log k) k^{\max(h(q + \varepsilon/2), h(q - \varepsilon/2)) - z} \ll (z/2)^{k+1} k^{2 - z}.
\]
If $|a - q| \ge \varepsilon / 2$, then $\delta \sqrt{\log k} \gg \sqrt{\log k}$.

Note that $h(2 + t) \le 2 - \frac{(t-2)^2}{8}$ on $t \in [0, 3]$.

Hence for $z \in [\varepsilon, 1 - \varepsilon] \cup [1 + \varepsilon, 2 - \varepsilon]$ we have
\[
z^{k + 1} \sum_{j = 0}^{k} |R_{k, j}| \ll
(z/2)^{k+1} k^{2 - z} \left( 1 + \sum_{n = 1}^{\infty} k^{-\frac{(n / \sqrt{\log k})^2}{8}} \right) \ll (z/2)^{k+1} k^{2 - z}.
\]

If $z = 1$ or $z = 2$, then
\[
z^{k + 1} \sum_{j = 0}^{k} |R_{k, j}| \ll 
(z/2)^{k+1} k^{2 - z} \sqrt{\log k} \left( 1 + \sum_{n = 1}^{\infty} \frac{k^{-\frac{(n / \sqrt{\log k})^2}{8}}}{n} \right) \ll 
(z/2)^{k+1} k^{2 - z} \sqrt{\log k}.
\]

If $z \ge 1 + \varepsilon$, then Lemma \ref{I_k = I_k'} gives us
\[
z^{k + 1} \sum_{j = 0}^k |I_{k, j} - I_{k, j}'| \ll (z/2)^{k}.
\]
And for $z = 1$ Lemma \ref{I_k = I_k'} gives us
\[
z^{k + 1} \sum_{j = 0}^k |I_{k, j} - I_{k, j}'| \ll (z/ 2)^k k.
\]

Note that for $1 + \varepsilon \le z \le 2$ the contribution of the residue at $z + 1$ is small compared to the error term and hence
\[
\gamma_k(z) = \frac{z^k}{\Gamma(z)} \left( \sum_{j = 0}^{k - z \log k} 2^{-j} \frac{G_2(1,z)}{z^{k - j + 1}} + 
O((\log k) (2/z)^{1.2 \log k}) \right) \, + O((z/2)^{k + 1} k^{2 - z} A_z).
\]

where $A_z = \sqrt{\log k}$ for $z \in \mathbb{Z}$ and $A_z = 1$ otherwise. 

Note that for $z < 2 - \varepsilon$
\[
z^{k + 1} \sum_{j = k - z \log k}^{\infty} \frac{2^{-j}}{z^{k - j + 1}} \ll_{\varepsilon} \left( \frac{z}{2} \right)^{k} \left( \frac{2}{z} \right)^{z \log k} \ll \left( \frac{z}{2} \right)^{k} k^{z \log(2/z)} \ll \left( \frac{z}{2} \right)^{k} k^{2 - z}.
\]

Hence for $1 + \varepsilon < z < 2 - \varepsilon$
\[
\gamma_{k}(z) = \frac{G_{2}(1, z)\left(1 - \frac{z}{2} \right)^{-1}}{\Gamma(z + 1)} + O((z/2)^{k} k^{2 - z}).
\]

And for $z = 2$
\[
\gamma_k(2) = \frac{1}{2} \sum_{j = 0}^{k - 2 \log k} G_2(1, 2) + O\left( \sqrt{\log k} \right) = \frac{G_2(1, 2)}{2} \left( k - 2 \log k + O(\sqrt{\log k}) \right).
\]

For $z = 1$ the residue at $z + 1$ now is greater than the error term. We get
\[
\gamma_k(1) = \frac{G_{2}(1, 1)\left(1 - \frac{1}{2} \right)^{-1}}{\Gamma(2)} - \frac{z^k}{\Gamma(z)} 
\left( \sum_{j = 0}^{k - 2\log k} (j \log 2) 2^{-j} \frac{G_2(1,2)}{2^{k - j + 1}}\right) + O((1/2)^{k} k \sqrt{\log k}).
\]

Thus
\[
\gamma_k(1) = 1 - 2^{-k} \frac{G_2(1, 2) (\log 2)}{4} (k^2 - 4k \log k + O(k \sqrt{\log k})).
\]

For $\varepsilon < z < 1 - \varepsilon$ we will use $I_{k, j}' + I_{k, j}''$ as an approximation to $I_{k, j}$. In the similar way we obtain
\[
I_{k, j}'' = \frac{1}{2\pi i} \oint_{|w| = \epsilon} \frac{G_2^{(1, 0)}(1, w)}{w^{k - j + 1}}
h_{j, z+1}(w) \, dw.
\]

Using the same analysis we get
\[
\frac{z^k}{\Gamma(z)} \sum_{j = 0}^{k} I_{k, j}'' = 
\frac{(z/2)^k}{\Gamma(z)} G_2^{(1, 0)}(1, z + 1)  \left( \sum_{j = 0}^{\infty} \frac{2^{k-j}}{(z + 1)^{k - j + 1}} + O\left( \left(\frac{2}{z + 1} \right)^{(z + 1) \log k} \right) \right) + 
\]
\[
O((z/2)^{k+1} k^{1 - z} \sqrt{\log k}) =
G_2^{(1, 0)}(1, z + 1) \left( \frac{z}{z+1} \right)^k \frac{\left(1 - \frac{z+1}{2} \right)^{-1}}{\Gamma(z)(z+1)}
+  O((z/2)^{k+1} k^{1 - z} \sqrt{\log k}).
\]

And
\begin{multline*}
\frac{z^k}{\Gamma(z)} \sum_{j = 0}^{k} I_{k, j}' = 
\frac{G_{2}(1, z)\left(1 - \frac{z}{2} \right)^{-1}}{\Gamma(z + 1)} + O((z/2)^{k} k^{2 - z}) \, - \\
\frac{(z/2)^k}{\Gamma(z)(z+1)} G_2(1, z+1) (\log 2) \left( \sum_{j = 0}^{k - (z+1) \log k} \left( \frac{2}{z+1} \right)^{k - j} j \right).
\end{multline*}

Denote $r = k - j$, then
\[
\sum_{r = 0}^{(z+1) (\log k)} \left( \frac{2}{z+1} \right)^{k - j} j \ll k^{(z + 1) \log(2/(z+1)) + 1} \ll k^{2 - z}.
\]
And for $\varepsilon < z < 1 - \varepsilon$
\[
\sum_{j = 0}^{k} \left( \frac{2}{z+1} \right)^{k - j} j =
\frac{\left( \frac{2}{z+1} \right)^{k+1}}{\left( \frac{2}{z+1} - 1 \right)^2} + O(k).
\]

Putting all things together we obtain for $\varepsilon < z < 1 - \varepsilon$
\begin{multline*}
\gamma_k(z) = 
\frac{G_{2}(1, z)\left(1 - \frac{z}{2} \right)^{-1}}{\Gamma(z + 1)} + O((z/2)^k k^{2 - z}) + \\
\left( \frac{z}{z+1} \right)^{k} \frac{2}{\Gamma(z) (1 - z^2)} \left( G_2^{(1, 0)}(1, z + 1) - (\log 2) \left( \frac{1 + z}{1 - z}\right) G_2(1, z + 1) \right).
\end{multline*}

Recall the equality (\ref{G_y(1, w)}) and also note that
\[
G_2^{(1, 0)}(1, w) = G_2(1, w) \left( w \log 2 + \sum_{p > 2} \left( \frac{w \log p}{p - 1} - \frac{w \log p}{p - w} \right) + w \gamma \right).
\]
Hence for $\varepsilon < z < 1 - \varepsilon$
\[
\gamma_k(z) = G(z) + \left( \frac{z}{z + 1} \right)^k \frac{2 G_2(1, z + 1)}{\Gamma(z) (1 - z)}\left(\gamma -\frac{z \log 2 }{1 - z} - \sum_{p > 2} \frac{z \log p}{(p - 1) (p - z - 1)} \right) + O \left( \left( \frac{z}{2} \right)^k k^{2 - z} \right).
\]

The result follows.

The uniform approximation follows from the same analysis. 
\end{proof}

\subsection{Proof of Theorem \ref{tightness}}
\begin{proof}[\unskip\nopunct]
Let $p_1, p_2, \ldots$ be the odd prime numbers in increasing order. Let $C_i$ be a monotonically increasing sequence to be specified later. Denote $P_i \vcentcolon = \prod_{j \le i} p_j$. We take
\[
A_i \vcentcolon = \{ p_i b \, : \, (b, P_i) = 1, \Omega(b) = C_i \}, \qquad A = \bigcup_{i = 1}^{\infty} A_i.
\]
First let us prove that $A$ is primitive. Suppose that $a \in A_i, a' \in A_j$ and $a | a'$, $a \ne a'$.  If $j > i$,  then $(p_i, a') = 1$ which gives a contradiction, since $p_i | a$. Hence $j \le i$. This implies that $\Omega(a') \le \Omega(a)$, which again gives a contradiction with $a | a'$, $a \ne a'$. Thus $A$ is primitive.

Now we choose $C_i \,\, (i \ge 1)$, so that three conditions are satisfied:
\begin{flalign*}
&\text{1) } \text{For } z \in (0, 2 - 2^{-n}] \text{ we have } 
\qquad \qquad \sum_{a \in A_i} \frac{\psi(a) z^{\Omega(a)}}{a (\log a)^z} \ge 2^{i}, \\
&\text{2) } f_2(A_i) \ge 2^i, &&\\
&\text{3) } \text{For } z \in [2 + 2^{-n}, 2^n] \text{ we have } f_z(A_i) \ge 2^{i}. &&
\end{flalign*}

Obviously if we could choose such increasing sequence $C_i$, then $A$ will satisfy all conditions of Theorem \ref{tightness}, since $A_i$ are disjoint. Thus it is enough to prove that for fixed $i$ conditions 1), 2), 3) are satisfied if $C_i$ is sufficiently large. 

Using inclusion-exclusion principle and Theorem \ref{gamma_k(z) main thm} we obtain that as $C_i$ tends to infinity
\begin{equation} \label{f_z(A_i) sim}
f_z(A_i) \sim 
\begin{cases}
G(z) \frac{z}{p_i} \prod_{j \le i} \left( 1 - \frac{z}{p_j} \right), \quad &\text{uniformly in } z \in (0, 2 - 2^{-i}] \\
\frac{d_2}{2} C_i \frac{2}{p_i} \prod_{j \le i} \left( 1 - \frac{2}{p_j} \right), \quad &\text{if } z = 2.
\end{cases}
\end{equation}
Hence for all sufficiently large $C_i$ we have $f_2(A_i) > 2^{i}$ and condition 2) is satisfied. 

Let 
\[
m_i \vcentcolon = \inf_{z \in (0, 2 - 2^{-i}]} G(z) \frac{z}{p_i} \prod_{j \le i} \left( 1 - \frac{z}{p_j} \right).
\]
Equation (\ref{f_z(A_i) sim}) implies that there exist $N_1$ such that for all $C_i \ge N_1$ it is $f_z(A_i) > m_i / 2$ for all $z \in (0, 2 - 2^{-i}]$.

There exist $N_2$ such that for each $n \ge N_2$ we have $\psi(n) > 2^{i + 1} m_i^{-1}$. Hence for $C_i > \max(N_1, N_2)$ condition 1) is satisfied.

Finally we note that $2^{C_i} p_i \in A_i$ and
\[
\lim_{C_i \to \infty} f_z(2^{C_i} p_i) = \infty, \quad \text{uniformly in } \, z \in [2 + 2^{-n}, 2^n].
\]
Hence condition 3) is satisfied for $C_i$ sufficiently large. 

\end{proof}

\subsection{Proof of Proposition \ref{H-L Tauber}}
\begin{proof}[\unskip\nopunct]
Let
\[
S(z) \vcentcolon =  \sum_{n = 2}^{\infty} \frac{h(\Omega)}{n (\log n)^z} = 
\sum_{k = 1}^{\infty} h(k) \gamma_k(z) z^{-k}.
\]
The last equality holds since $\sum_{n = 2}^{\infty} \frac{h(\Omega(n))}{n (\log n)^z}$ is absolutely convergent. Since $\gamma_k(z) \sim G(z)$ uniformly in $z \in [1, 1.5]$, we have $\sum_{k \le x} h(k) \gamma_k(z) \sim c \, G(z) x$ uniformly in $z \in [1, 1.5]$. Integrating by parts we obtain that as $z$ tends to $1+$
\[
S(z) \sim c \, G(z) z^{-1} \left(1 - \frac{1}{\log z^{-1}} \right).
\]
Since $\lim_{z \to 1+} G(z) = 1$, it follows that
\[
\lim_{z \to 1+} (z - 1)S(z) = c.
\]
\end{proof}

\section{$f_z(A)$ upper bounds: Theorem \ref{U(z) main}}
In this section we follow \cite{Lichtman} to derive upper bounds on $f_z(A)$.
By $L$-primitive set in this section we assume an $L$-primitive set with respect to increasing order. 

First let us introduce some notation.

For an $L$-primitive set $A$ we denote $L_A \vcentcolon = \bigcup_{a \in A} L_a$, where $L_a = \{ ab \in \mathbb{N} : P(a) \le p(b) \,\, \text{or} \,\,  b = 1\}$. Note that by Lemma \ref{L_a intersection} this is a disjoint union. 

Let us denote 
\begin{equation} \label{d_z(L_A) definition}
d_z(L_a) \vcentcolon = \frac{z^{\Omega(a)}}{a} \prod_{p < P(a)} \left(1 - \frac{z}{p} \right), \quad \quad
d_z(L_A) \vcentcolon = \sum_{a \in A} d_z(L_a).
\end{equation}

Theorem \ref{main} implies that $d_z(L_A) \le 1$.

\begin{lemma} \label{trivial1}
Assume that $A, B$ are finite $L$-primitive sets and $A \subset L_B$, then
\[
d_z(L_B) \ge d_z(L_A).
\]
\end{lemma}
\begin{proof}
    Let us take $M = \prod_{m \in A \cup B} m!$ and let $g$ be a completely multiplicative function such that 
    \[
    g(p) = 
    \begin{cases}
    z/p, & \text{if} \,\,\, p | M, \\
    0, & \text{otherwise}.  
    \end{cases}
    \]
    As in the proof of Theorem \ref{main} we have
    \[
    \left( \sum_{m \in \mathbb{N}} g(m) \right) d_z(L_B) = \sum_{b \in B} \sum_{m \in L_b} g(m) \ge \sum_{a \in A} \sum_{m \in L_a} g(m) = \left( \sum_{m \in \mathbb{N}} g(m) \right) d_z(L_A).
    \]
    And the desired inequality follows.
\end{proof}

\begin{rem}
We can view $d_z(L_a)$ as a density of $L_a$ in the following sense. More generally for $M \subset \mathbb{N}$ we can set
\[
d_z(M) \vcentcolon = \lim_{x \to \infty} \frac{\sum_{n \in M \cap [1, x]} z^{\Omega(n)}}{\sum_{n \le x} z^{\Omega(n)}}.
\]
This gives another approach to prove Lemma \ref{trivial1} and Theorem \ref{main} for the case $f(n) = z^{\Omega(n)}/n$ and primes in increasing order.
\end{rem}

Let us denote
\begin{equation} \label{C_z}
C_z \vcentcolon = \prod_{p} \left(1 - \frac{z}{p} \right)^{-1} \left(1 - \frac{1}{p} \right)^{z}.
\end{equation}

Let 
\[
\mu_x(z) \vcentcolon = \left( e^{\gamma} (\log x) \prod_{p < x} \left(1 - \frac{1}{p} \right) \right)^z \prod_{p \ge x} \left(1 - \frac{z}{p} \right)^{-1} \left(1 - \frac{1}{p} \right)^z = e^{\gamma z} C_z (\log x)^{z} \prod_{p < x} \left( 1 - \frac{z}{p} \right).
\]

Mertens' third theorem implies that $\mu_x(z) \sim 1$ as $x$ tends to infinity uniformly in $z \in [0, 2]$.

For $q \in \mathbb{P}$ and $x \in \mathbb{R}$ we define
\[
m_q(z) \vcentcolon = \inf_{\substack{p \ge q \\ p \in \mathbb{P}}} \mu_p(z), \quad \quad 
M_x(z) \vcentcolon = \sup_{\substack{y \ge x \\ y \in \mathbb{R}}} \mu_y(z), \quad \quad r_q(z) \vcentcolon = \sup_{\substack{p \ge q \\ p \in \mathbb{P}}} \frac{M_p(z)}{\mu_p(z)}.
\]
Obviously $r_q(z) \le M_q(z) / m_q(z)$. Also note that for a prime $q$
\[
M_q(z) = \sup_{\substack{p \ge q \\ p \in \mathbb{P}}} \mu_p(z).
\]

For a prime number $q$ we have
\[
f_z(q) = \frac{z}{q (\log q)^z} = \frac{z}{q} \frac{e^{\gamma z} C_z}{\mu_q(z)} \prod_{p < q} \left( 1 - \frac{z}{p} \right) = \frac{e^{\gamma z} C_z}{\mu_q(z)} d_z(L_q).
\]

Denote $A_n \vcentcolon = A \cap L_n$.

\begin{lemma} \label{L1bound}
    Let $A$ be an $L$-primitive set. Take $\nu \ge 0$, an integer $n \notin A$ and denote $q = P(n)$. If $P(a)^{1 + \nu} \le a$ for all $a \in A_n$, then
    \[
    f_z(A_n) \le \frac{e^{\gamma z} C_z}{m_q(z)} \frac{d_z(L_{A_n})}{(1 + \nu)^z},
    \]
    where $C_z$ is defined by (\ref{C_z}).
\end{lemma}
\begin{proof}
    $P(a)^{1 + \nu} \le a$ implies
    \[
    f_z(a) = \frac{z^{\Omega(a)}}{a (\log a)^z} \le 
    \frac{z^{\Omega(a)}(1 + \nu)^{-z} }{a (\log P(a))^z} = 
    \frac{e^{\gamma z} C_z}{\mu_{P(a)}(z)} \frac{z^{\Omega(a)}}{a (1+\nu)^z} \prod_{p < P(a)} \left(1 - \frac{z}{p} \right) = \frac{e^{\gamma z} C_z}{\mu_{P(a)}(z)} \frac{d_z(L_a)}{(1 + \nu)^z}.
    \]
    We have $\mu_{P(a)}(z) \ge m_{P(a)}(z) \ge m_q(z)$. Hence
    \[
    f_z(A_n) = \sum_{a \in A_n} f_z(a) \le \sum_{a \in A_n} \frac{e^{\gamma z} C_z}{m_q(z)} \frac{d_z(L_a)}{(1 + \nu)^z} = \frac{e^{\gamma z} C_z}{m_q(z)} \frac{d_z(L_{A_n})}{(1 + \nu)^z}.
    \]
\end{proof}

Denote $a^* \vcentcolon = a / P(a)$ and
\[
C_a^{\nu} \vcentcolon = \{c \in \mathbb{N} : [p(c), P(c)] \subset [P(a^*), P(a^*)^{1/ \sqrt{\nu}}) \}.
\]
\begin{lemma}[Lichtman] \label{Lichtman lemma}
Let $A$ be a primitive set of composite numbers and $\nu \in (0, 1)$. If $P(a)^{1 + \nu} > a$ for all $a \in A$, then the sets $L_{ac}$ ranging over $a \in A, c \in C_a^{\nu}$ are pairwise disjoint. In particular the set
$\{a c : a \in A, c \in C_{a}^{\nu} \}$ is
$L$-primitive.
\end{lemma}

\begin{proof}
    See \cite[Lemma 3.1]{Lichtman}.
\end{proof}

\begin{lemma} \label{LA_n bound}
    Let $A$ be a finite primitive set. Take $\nu \in (0, 1)$, an integer $n > 1$ with $n \notin A$ and denote $q = P(n)$. If $P(a)^{1 + \nu} > a$ for all $a \in A_n$, then 
    \[
    d_z(L_{A_n}) \le \nu^{z/2} r_q(z) d_z(L_n).
    \]
\end{lemma}

\begin{proof}
    Without loss of generality assume that $A = A_n$. Let $a \in A, c \in C_a^{\nu}$. We have $p(c) \ge P(a^*) \ge P(n)$. Hence $ac \in L_n$. Thus
    \[
    L_n \supset \bigcup_{a \in A} \bigcup_{c \in C_{a}^{\nu}} L_{ac}.
    \]
    Lemma \ref{Lichtman lemma} implies that this is a disjoint union.
    Also $P(ac) = P(a)$ and hence by Lemma \ref{trivial1} 
    \[
    d_z(L_n) \ge \sum_{a \in A} \sum_{c \in C_a^{\nu}} d_z(L_{ac}) = 
    \sum_{a \in A} d_z(L_a) \sum_{c \in C_a^{\nu}} \frac{z^{\Omega(c)}}{c}.
    \]
    \[
    \sum_{c \in C_a^{\nu}} \frac{z^{\Omega(c)}}{c} = 
    \prod_{p \in [P(a^*), P(a^*)^{1/\sqrt{\nu}})} \left(1 - \frac{z}{p} \right)^{-1} = \prod_{p < P(a^*)^{1/\sqrt{\nu}}} \left(1 - \frac{z}{p} \right)^{-1} \prod_{p < P(a^*)} \left(1 - \frac{z}{p} \right) =
    \]
    \[
    \frac{(\log P(a^*)^{1/\sqrt{\nu}})^z}{\mu_{P(a^*)^{1/\sqrt{\nu}}}(z)} 
    \frac{\mu_{P(a^*)}(z)}{(\log P(a^*))^z} = 
    \nu^{-z/2} \frac{\mu_{P(a^*)}(z)}{\mu_{P(a^*)^{1/\sqrt{\nu}}}(z)} \ge \nu^{-z/2} \frac{\mu_{P(a^*)}(z)}{M_{P(a^*)}(z)} \ge \nu^{-z/2} r_q(z)^{-1}.
    \]
    
    This gives 
    \[
    d_z(L_n) \ge \nu^{-z/2}  r_q(z)^{-1} \sum_{a \in A} d_z(L_a) = 
    \nu^{-z/2}  r_q(z)^{-1} d_z(L_A).
    \]
\end{proof}

\begin{lemma} \label{c_i, D_i}
For $k \ge 1$, let $c_0 \ge c_1 \ge \ldots \ge c_k \ge 0$ and $0 = D_0 \le D_1 \le \ldots \le D_k$. If $d_1, \ldots, d_k \ge 0$ satisfy $\sum_{j \le i} d_j \le D_i$ for all $i \le k$, then we have
\[
\sum_{i \le k} c_i d_i \le \sum_{i \le k} c_i (D_i - D_{i - 1}). 
\]
\end{lemma}
\begin{proof}
    See \cite[Lemma 4.1]{Lichtman}.
\end{proof}

Let us denote
\[
b_q(z) \vcentcolon = I(z) \frac{r_q(z)}{m_q(z)} \mu_q(z), \qquad
I(z) \vcentcolon = \frac{z}{2} \int_0^1 \frac{ \nu^{z/2 - 1} \, d\nu}{(1 + \nu)^z}.
\]
\begin{prop} \label{U(z) main tool}
    For any primitive set $A$, and any integer $n \notin A$ with $q = P(n)$,
    \[
    f_z(A_n) \le  \frac{q z^{\Omega(n) - 1}}{n} b_q(z)  f_z(q).
    \]
\end{prop}

\begin{proof}
    We may assume that $A = A_n$ is finite and then take a limit. All elements of $A$ are composite, since they are divisible by $n$ and $n \notin A$.

    Take $k \ge 1$ and any sequence $0 = \nu_0 < \nu_1 < \ldots < \nu_k = 1$, and partition the set $A = \bigcup_{0 \le i \le k} A_{(i)}$, where $A_{(k)} = \{ a \in A : P(a)^2 \le a \}$ and for $0 \le i < k$,
    \[
    A_{(i)} = \{ a \in A : P(a)^{1 + \nu_i} \le a  < P(a)^{1 + \nu_{i + 1}} \}.
    \]
    Application of Lemma \ref{L1bound} to each $A_{(i)}$ gives
    \[
    f_z(A) = \sum_{0 \le i \le k} f_z(A_{(i)}) \le 
    \frac{e^{\gamma z} C_z}{m_q(z)} \sum_{0 \le i \le k} \frac{d_z(L_{A_{(i)}})}{(1 + \nu_i)^z}.
    \]
    For each $j < k$ denote $A^{(j)} = \bigcup_{0 \le i \le j} A_{(i)} = \{a \in A : a < P(a)^{1 + \nu_{j + 1}} \}$. Lemma \ref{LA_n bound} implies
    \[
    \sum_{0 \le i \le j} d_z(L_{A_{(i)}}) = d_z(L_{A^{(j)}}) \le \nu_{j + 1}^{z/2} r_q(z) d_z(L_n).
    \]
    Trivially $\sum_{0 \le i \le k} d_z(L_{A_{(i)}}) = d_z(L_A) \le d_z(L_n) \le r_q(z) d_z(L_n)$. 
    Let $c_i = (1 + \nu_i)^{-z}$, $d_i = d_z(L_{A_{(i)}})$, $D_i = \nu_{i + 1}^{z/2} r_q(z) d_z(L_n)$ (and we set $\nu_{k + 1} = \nu_k$, so that $D_k - D_{k - 1} = 0$). Then by Lemma \ref{c_i, D_i}
    \[
    \sum_{0 \le i \le k} \frac{d_z(L_{A_{(i)}})}{(1 + \nu_i)^z} = 
    \sum_{0 \le i \le k} c_i d_i \le \sum_{0 \le i \le k} c_i (D_i - D_{i - 1}) =
    r_q(z) d_z(L_n) \sum_{0 \le i \le k} \frac{\nu_{i + 1}^{z/2} - \nu_{i}^{z/2}}{(1 + \nu_i)^{z}}. 
    \]
    Hence
    \[
    f_z(A) \le e^{\gamma z} C_z \frac{r_q(z)}{m_q(z)} d_z(L_n) 
    \sum_{0 \le i \le k} \frac{\nu_{i + 1}^{z/2} - \nu_{i}^{z/2}}{(1 + \nu_i)^{z}}. 
    \]
    We have $0 = \nu_0 < \nu_1 \le \ldots < \nu_k = 1$. Let $\nu_i = \frac{i}{k}$ and let $k$ tend to infinity. We obtain for some $\xi_i \in [\nu_{i - 1}, \nu_i]$:
    \[
    \lim_{k \to \infty} \sum_{1 \le i \le k}  \frac{\nu_{i}^{z/2} - \nu_{i - 1}^{z/2}}{(1 + \nu_{i - 1})^{z}} = 
    \lim_{k \to \infty} \sum_{1 \le i \le k}  (z/2) \xi_i^{z/2 - 1} \frac{\nu_{i} - \nu_{i - 1}}{(1 + \nu_{i - 1})^{z}} = \int_0^1 \frac{d(\nu^{z/2})}{(1 + \nu)^z}.
    \]
    It follows that
    \[
    f_z(A_n) \le e^{\gamma z} C_z \frac{r_q(z)}{m_q(z)} \left( \int_0^1 \frac{d(\nu^{z/2})}{(1 + \nu)^z} \right) d_z(L_n).
    \]
    Finally
    \[
    e^{\gamma z} C_z \, d_z(L_n) = e^{\gamma z} C_z \frac{z^{\Omega(n)}}{n} \prod_{p < q} \left(1 - \frac{z}{p} \right) = 
    \mu_q(z) \frac{z}{q (\log q)^z} \frac{q z^{\Omega(n) - 1}}{n} = \mu_q(z) f_z(q) \frac{q z^{\Omega(n) - 1}}{n}.
    \]
    The result follows.
\end{proof}

\begin{proof}[Proof of Theorem \ref{U(z) main}C]
    Proposition \ref{U(z) main tool} implies that $f_z(A_{q}) \le b_q(z) f_z(q)$. 
    Hence if $b_q(z) \le 1$, then $q$ is Erdős $z$--strong.

    We have $b_q(z) \le I(z) \left(\frac{M_q(z)}{m_q(z)} \right)^2$. By substitution we obtain 
    \[
    I(z) = \int_0^1 S(\tau, z) d \tau, \qquad
    S(\tau, z) \vcentcolon = (1 + \tau^{2/z})^{-z}.
    \]
    Note that for $\tau \in (0, 1]$, $z \in (0, 2]$ we have $0 < S(\tau, z) < 1$.
    Continuity of $S(\tau, z)$ implies that for each $\varepsilon > 0$ there exist $\delta > 0$ such that $S(\tau, z) < 1 - \delta$ for $\tau \in [1/2, 1], z \in [\varepsilon, 2]$.
    Thus $I(z) < 1 - \delta/2$ for $z \in [\varepsilon, 2]$.
    Since $\mu_q(z) \sim 1$ as $q$ tends to infinity uniformly in $z \in [0, 2]$, we obtain Theorem \ref{U(z) main}C.
\end{proof}

    Now we want to evaluate $b_q(z)$. In what follows we present some bounds which were used to perform the proof of Theorem \ref{U(z) main}A by computation on computer. 

    \begin{lemma} \label{b_q(z)' bound}
        For $q \ge 3$ and $z \in (0, 2]$ we have
        \[
        \partial_{+}b_q(z) \le 3.54 \, b_q(z),
        \]
        where $\partial_{+}$ denotes the right derivative.
    \end{lemma}

    \begin{proof}
    For $\tau \in (0, 1)$ the function $S(\tau, z)$ monotonically decreases in $z$.
    Hence $I(z)$ is also monotonically decreasing. Therefore
    \begin{multline*}
    \partial_{+} b_q(z) = b_q(z) \partial_{+}(\log b_q(z)) \le 
    b_q(z) \left( (\log I(z))' + 4 \sup_{x \in [q, \infty)} |\mu_x(z)'| \right) \le
    4b_q(z) \sup_{x \in [q, \infty)} |\mu_x(z)'| \le
    \\
    b_q(z) \left(
    4 \sum_{p \ge 3} \left|\left(z\log\left(1 - \frac{1}{p} \right) - \log \left(1 - \frac{z}{p} \right) \right)' \right|  
    + 4 \sup_{x \in [q, \infty)} \left|\log \mu_x(1)\right|
    \right).
    \end{multline*}
    In \cite[Lemma 2.4]{Lichtman} it is proved that $M_3(1) \le 1 + \frac{1}{2 \log(2 \cdot 10^9)^2}$ and $m_3(1) \ge 0.925$. Thus
    \[
    \sup_{x \in [3, \infty)} \left|\log \mu_x(1)\right| \le |\log(0.925)| \le 0.078.
    \]

    We have
    \[
    \sum_{p \ge 3} \left|\left(z\log\left(1 - \frac{1}{p} \right) - \log \left(1 - \frac{z}{p} \right) \right)' \right| = \sum_{p \ge 3} \left| \sum_{k \ge 2} \frac{z^{k - 1} - 1/k}{p^k} \right| \le  
    \sum_{p \ge 3} \left| \sum_{k \ge 2} \frac{2^{k - 1} - 1/k}{p^k} \right| \le
    \]
    \[ 
     \sum_{p \ge 3} \left( \frac{2}{p^2} \left(1 - \frac{2}{p} \right)^{-1} + \log\left( 1 - \frac{1}{p} \right) + \frac{1}{p}  \right) \le 0.805.
    \]
    Hence
    \[
    \partial_{+} b_q(z) \le 4 b_q(z)(0.805 + 0.078) \le 3.54 \, b_q(z).
    \]
    \end{proof}
    \begin{cor} \label{computation precision}
        If $0 < a < b < 1$, $(b - a) \le 1/4$ and $b_q(a) \le C$, then $b_q(z) \le C \left( 1 + \frac{3.54 (b - a)}{1 - 3.54 (b - a)} \right)$ on $z \in [a, b]$.
    \end{cor}
    \begin{proof}
        If $b_q(z) = C + d$, $d > 0$ for some $z$, then mean value theorem implies that there exist $z_0 \in [a, z]$, such that $b_q(z_0)' \ge d / (b - a)$ and $b_q(z) \ge C$. 
        In view of Lemma \ref{b_q(z)' bound} this gives 
        $\frac{d}{b - a} \le 3.54 (C + d)$.
    \end{proof}
    This corollary allows us to give upper bounds for $b_q(z)$ on small intervals by evaluating it at a fixed $z$.

    To give an upper bound for $b_q(z)$ we need an upper and a lower bounds on $\mu_x(z)$.

    \cite[Lemma 2.4]{Lichtman} implies that for $ q \ge 300$
    \begin{equation} \label{eq1comp}
    1 - \frac{1}{2 (\log q)^2} \le \mu_q(1) \le 1 + \min \left( \frac{1}{2 (\log 2 \cdot 10^9)^2}, \frac{1}{2 (\log q)^2} \right).
    \end{equation}
    For $z \le 1$ we have
    \begin{equation} \label{eq2comp}
    0.9998^z \prod_{x \le p \le 300} \left(1 - \frac{z}{p} \right)^{-1} \left( 1 - \frac{1}{p} \right)^z \le
    \prod_{x \le p \le 300} \left(1 - \frac{z}{p} \right)^{-1} \left( 1 - \frac{1}{p} \right)^z  \exp \left( - \sum_{p > 300} \sum_{k \ge 2} \frac{z}{k p^k}\right) \le 
    \end{equation}
    \[
    \prod_{p \ge x} \left( 1 - \frac{z}{p} \right)^{-1} \left( 1 - \frac{1}{p} \right)^{z} \le 
    \prod_{x \le p \le 300} \left(1 - \frac{z}{p} \right)^{-1} \left( 1 - \frac{1}{p} \right)^z.
    \]
    For $z \ge 1$ we have
    \begin{equation} \label{eq3comp}
    \prod_{x \le p \le 300} \left(1 - \frac{z}{p} \right)^{-1} \left( 1 - \frac{1}{p} \right)^z \le 
    \prod_{p \ge x} \left( 1 - \frac{z}{p} \right)^{-1} \left( 1 - \frac{1}{p} \right)^{z} \le
    \end{equation}
    \[
    \prod_{x \le p \le 300} \left(1 - \frac{z}{p} \right)^{-1} \left( 1 - \frac{1}{p} \right)^z \prod_{p > 300} \left(1 - \frac{2}{p} \right)^{-1} \left( 1 - \frac{1}{p} \right)^2 \le
    1.0005 \prod_{x \le p \le 300} \left(1 - \frac{z}{p} \right)^{-1} \left( 1 - \frac{1}{p} \right)^z.
    \]
    Note that
    \[
    \mu_x(z) = \mu_x(1)^z \, \prod_{p \ge x} \left( 1 - \frac{z}{p} \right)^{-1} \left( 1 - \frac{1}{p} \right)^{z}.
    \]
    We run the computer program which uses Corollary \ref{computation precision} and inequalities (\ref{eq1comp}), (\ref{eq2comp}), (\ref{eq3comp}) to obtain upper bounds on $b_q(z)$ on intervals. 

    We obtain that for $z \in [0.44, 2]$ we have $b_q(z) < 1$ for all odd primes. This proves Theorem \ref{U(z) main}A.

    If $2 \in A$ and $A$ is primitive, then Theorem \ref{U(z) main}A implies that $f_z(A) \le \gamma_1(z)$ on $z \in [0.44, 2]$.
    Let us assume that $2 \notin A$. 
    In \cite[Theorem 4.4]{Lichtman} it is proved that in this case $f_1(A) < 1.60 < \gamma_1(1)$. 
    Theorem \ref{U(z) main}B follows since all bounds we derived are continuous in $z$ and one can follow the proof of \cite[Theorem 4.4]{Lichtman}, which will work in some neighborhood of $1$. 

    \begin{rem}
        One can do a computation to derive some explicit interval $a < 1 < b$, for which  $U(z) = \gamma_1(z)$. 
        But surprisingly the presented adaptation of the method from \cite{Lichtman} does not seem to work for $z$ that are sufficiently close to $0$. We can't even prove that all sufficiently large primes are Erdős $z$-strong on $(0, 2]$. The reason is that $\lim_{z \to 0} I(z)^{1/z} = 1$ and $\mu_x(z)$ tends to $1$ as $x$ tends to infinity, but the convergence is not sufficiently fast. 
    \end{rem}

\section{Primitive density: Theorem \ref{pr_dens}}

\begin{lemma} \label{finite}
    Let $A$ be a primitive set such that $P(a) \le N$ for every $a \in A$. Then $A$ is finite.
\end{lemma}

\begin{proof}
    Let $p_1, p_2, \ldots, p_k$ be the sequence of all primes that are not greater than $N$. 

    On $\mathbb{Z}_{\ge 0}^k$ we can define a partial order by $(a_1, a_2, \ldots, a_k) \le (b_1, b_2, \ldots, b_k)$ iff $a_i \le b_i$ for every $i$. 

    The map $\alpha : \mathbb{Z}_{\ge 0}^k \to 
    \{ n \in \mathbb{N} \, : \, P(n) \le N \}$, $\alpha(a_1, a_2, \ldots, a_k) = p_1^{a_1}p_2^{a_2}\ldots p_k^{a_k}$
    is an isomorphism of partially ordered sets. In particular it gives a one-to-one correspondence for their antichains.

    Hence it is enough to prove that each antichain of $\mathbb{Z}_{\ge 0}^k$ is finite. We will prove it by induction.

    For $k = 1$ the statement is obvious. Now suppose that it holds for $k - 1$. Let us prove it for $k$. Let $A$ be any antichain of $\mathbb{Z}_{\ge 0}^k$. Let $(a_1, a_2, \ldots, a_n) \in A$. Then
    \[
    A = \bigcup_{i = 1}^n \bigcup_{d = 0}^{a_i} A_{(i, d)}, \,\, \text{where} \, 
    A_{(i, d)} \vcentcolon = A \cap \{ (b_1, b_2, \ldots, b_k) \in \mathbb{Z}_{\ge 0}^k \, : \, b_i = d \}.
    \]
    It is easy to see that $\{ (b_1, b_2, \ldots, b_k) \in \mathbb{Z}_{\ge 0}^k \, : \, b_i = d \}$ are isomorphic to $\mathbb{Z}_{\ge 0}^{k-1}$ and $A_{(i, d)}$ are their antichains. Hence $A_{(i, d)}$ are all finite. Thus $A$ is finite.
\end{proof}

\begin{lemma} \label{Mert3}
If $0 < z < 2 - \delta$, then
\[
\prod_{p \le x} \left( 1 - \frac{z}{p} \right) = \frac{ C_z^{-1} e^{-\gamma z}}{(\log x)^z} \left(1 + O_{\delta}\left( \frac{1}{\log x} \right) \right),
\]
where $C_z$ is defined by (\ref{C_z}).
\end{lemma}
\begin{proof}
    In case $z = 1$ this is Mertens' third theorem. See, for example, \cite[p. 19, Theorem 1.12]{Tenenbaum}. 
    By Taylor expansion we deduce that
    \[
    C_z^{-1} = \prod_{p \le x} \left(1 - \frac{z}{p} \right) \left(1 - \frac{1}{p} \right)^{-z} (1 + O_{\delta}(x^{-1})).
    \]
    We have
    \[
    \prod_{p \le x} \left( 1 - \frac{z}{p} \right) = C_z^{-1} (1 + O_{\delta}(x^{-1})) \prod_{p \le x} \left(1 - \frac{1}{p} \right)^{z} =  \frac{C_z^{-1} e^{-\gamma z}}{(\log x)^z} \left(1 + O_{\delta}\left( \frac{1}{\log x} \right) \right).
    \]
\end{proof}

Recall (\ref{d_z(L_A) definition}) the definition of $d_z(L_A)$.

\begin{lemma} \label{prim_dens_upper}
Let $A$ be a primitive set such that $1 \notin A$. Then as $z$ tends to $0$
\[
f_z(A) = \sum_{a \in A} \frac{z^{\Omega(a)}}{a} \prod_{p < a}\left( 1 - \frac{z}{p }\right) + o(1) \le 
d_z(L_A) + o(1) =
\sum_{a \in A} \frac{z^{\Omega(a)}}{a (\log P(a))^z} + o(1).
\]
\end{lemma}

\begin{proof}

Lemma \ref{Mert3} shows that for any $\epsilon > 0$ there exist $N > 0, \delta > 0$ such that $\forall x > N, \forall z < \delta$ one has 
\[
\prod_{p < x} \left(1 - \frac{z}{p} \right) = \frac{(1 + \theta)^{-1}}{(\log x)^z}, \quad |\theta| < \epsilon.
\]

We have
\[
f_z(A) =  \left( \sum_{\substack{a \in A \\ a \le N}} + \sum_{\substack{a \in A \\ a > N}} \right) \frac{z^{\Omega(a)}}{a (\log a)^z}  = \Sigma_1 + \Sigma_2.
\]

For any fixed $N$ we have $\Sigma_1 = o(1)$. Hence for each $z < \delta$ we obtain
\[
f_z(A) = \Sigma_2 + o(1) = \sum_{a \in A} \frac{(1 + \theta_a)z^{\Omega(a)}}{a} \prod_{p < a} \left(1 - \frac{z}{p} \right) + o(1),
\]
where $|\theta_a| < \epsilon$ for each $a$. Thus for each $z < \delta$
\[
f_z(A) = (1 + \theta) \sum_{a \in A} \frac{z^{\Omega(a)}}{a} \prod_{p < a} \left(1 - \frac{z}{p} \right) + o(1), \quad |\theta| < \epsilon.
\]
But this holds for an arbitrary small $\epsilon$. Hence as $z$ tends to $0$
\[
f_z(A) = \sum_{a \in A} \frac{z^{\Omega(a)}}{a} \prod_{p < a} \left(1 - \frac{z}{p} \right) + o(1).
\]

The equality 
\[
d_z(L_A) = \sum_{a \in A} \frac{z^{\Omega(a)}}{a (\log P(a))^z} + o(1)
\]
can be proved in the same way as for each $N$ the set $\{a : P(a) < N\}$ is finite by Lemma \ref{finite}.

Finally 
\[
\sum_{a \in A} \frac{z^{\Omega(a)}}{a} \prod_{p < a} \left(1 - \frac{z}{p} \right) \le 
\sum_{a \in A} \frac{z^{\Omega(a)}}{a} \prod_{p < P(a)} \left(1 - \frac{z}{p} \right) = d_z(L_A).
\]
\end{proof}

\begin{lemma} \label{f_k(z) = 1 + o(1)}
For all $z \in (0, 2)$ we have  $d_z(L_{\mathbb{P}_k}) = 1$. And as $z$ tends to $0$ 
\[
\gamma_k(z) = d_z(L_{\mathbb{P}_k}) + o(1).
\]
\end{lemma}
\begin{proof}
    We analyse the proof of Theorem \ref{main}. Let $\preceq$ be the usual order $\le$ on $\mathbb{P}$. Let $A_N \vcentcolon = \{a \in \mathbb{P}_k : P(a) \le N \}$. Let $g_N(n)$ be a completely multiplicative function such that $g_N(p) = z/p$ if $p \le N$, $g_N(p) = 0$ otherwise. 
    
    As in the proof of Theorem \ref{main} we obtain
    \[
    \sum_n g_N(n) - \sum_{n \in \mathbb{N} \setminus L_{A_N}} g_N(n) = \sum_n g_N(n) \sum_{a \in A_N} g_N(a) \prod_{p < P(a)} (1 - g_N(p)) = \left( \sum_n g_N(n) \right) d_z(L_{A_N}).
    \]
    Obviously $\lim_{N \to \infty} d_z(L_{A_N}) = d_z(L_{\mathbb{P}_k})$. Thus to prove that $d_z(L_{\mathbb{P}_k}) = 1$ it is enough to show that
    \[
    \lim_{N \to \infty} \frac{\sum_{n \in \mathbb{N} \setminus L_{A_N}} g_N(n)}{\sum_n g_N(n)} = 0.
    \]
    
    We have
    \[
    \sum_{n \in \mathbb{N} \setminus L_{A_N}} g_N(n) = \sum_{l < k} \sum_{a \in \mathbb{P}_l : P(a) \le N } \frac{z^l}{a} \le \left( 1 +  \sum_{p \le N} \frac{z}{p} \right)^{k-1} \asymp_z z^{k-1} (\log \log N)^{k-1}.
    \]
    And
    \[
    \sum_n g_N(n) = \prod_{p \le N} \left( 1 - \frac{z}{p} \right)^{-1} \asymp (\log N)^z,
    \]
    which finishes the proof of $d_z(L_{\mathbb{P}_k}) = 1$.

    Now let us prove the second part of the Lemma.
   By Lemma \ref{prim_dens_upper} we have $\gamma_k(z) \le 
   d_z(L_{\mathbb{P}_k}) + o(1)$. Hence it enough to prove that $\gamma_k(z) \ge 
   d_z(L_{\mathbb{P}_k}) + o(1)$ as $z \to 0$. Note that if $\Omega(n) = k$, then $\log n \le k \log(P(n))$. Hence we have
   \[
   \gamma_k(z) = \sum_{\Omega(n) = k} \frac{z^k}{n (\log n)^z} \ge k^{-z} \sum_{a \in \mathbb{P}_k} \frac{z^{\Omega(a)}}{a (\log P(a))^z} = k^{-z} d_z(L_{\mathbb{P}_k}) + o(1) = d_z(L_{\mathbb{P}_k}) + o(1).
   \]

\end{proof}

\begin{proof}[Proof of Theorem \ref{pr_dens}]
    The inequalities $0 \le \underline{\eta}(A) \le \overline{\eta}(A)$ are trivial and we only have to prove that $\overline{\eta}(A) \le 1$. By Lemma \ref{prim_dens_upper} and Theorem \ref{main}
    \[
    f_z(A) \le d_z(L_{A}) + o(1) \le 1 + o(1)
    \]
    which proves Theorem \ref{pr_dens}A.

    Theorem \ref{pr_dens}B is essentially Lemma \ref{f_k(z) = 1 + o(1)}.

    To prove Theorem \ref{pr_dens}C we need to construct $A$ such that $\underline{\eta}(A) = 0, \, \overline{\eta}(A) = 1$. We  construct $A$ as a subset of $\mathbb{P}$.

    Let us define the sequences $z_n, a_n$ by induction. Let $a_0 = 0$. 
    For $n \ge 1$ let us define $z_n$ by conditions 
    $
    z_n < 2^{-n}, \,\,
    \sum_{p > a_{n-1}} \frac{z_n}{p(\log p)^{z_n}} > 1 - 2^{-(n + 1)}.
    $
    Such $z_n$ exists because $\eta(\mathbb{P}) = 1$.
    Now let us define $a_n$ by conditions
    $\sum_{a_{n-1} < p \le a_n} \frac{z_n}{p(\log p)^{z_n}} > 1 - 2^{-n}$, $a_n > a_{n-1}$.

    Take 
    \[
    A = \mathbb{P} \cap \bigcup_{n \ge 1} (a_{2n - 1}, a_{2n}]
    \]
    Denote
    \[
    d(z) = \sup_{0 < z' \le z} \left| \gamma_1(z')\right|.
    \]
    We know that $\lim_{z \to 0} d(z) = 1$.
    
    Then for $n \ge 1$, $f_{z_{2n}}(A) > 1 - 2^{-2n}$, $f_{z_{2n + 1}}(A) < d(2^{-2n-1}) - (1 - 2^{-2n - 1})$ and $\lim_{n \to \infty} z_n = 0$.
Therefore $\underline{\eta}(A) = 0, \, \overline{\eta}(A) = 1$.

    Now let us prove Theorem \ref{pr_dens}D. We have $A \subset \mathbb{P}_k$ and Dirichlet density of $A$ is $c$. Lemma \ref{sum to int} implies
    \[
    f_z(A) = \frac{z^k}{\Gamma(z)} \int_1^{\infty} \left( \sum_{a \in A} a^{-s} \right) (s - 1)^{z-1} ds.
    \]
    According to the conditions of the Theorem for any $\delta > 0$ there exist $\varepsilon > 0$, such that for $s \in (1, 1 + \varepsilon]$
    \[
    (c - \delta) \sum_{n \in \mathbb{P}_k} n^{-s} 
    \le
    \sum_{a \in A} a^{-s}
    \le
    (c + \delta) \sum_{n \in \mathbb{P}_k} n^{-s}.
    \]

    Note that
    \[
    \frac{z^k}{\Gamma(z)} \int_{1 + \varepsilon}^{\infty} (\zeta(s) - 1) (s - 1)^{z - 1} ds \ll_{\varepsilon} \frac{z^k}{\Gamma(z)} \int_{1}^{\infty} 2^{-s} (s - 1)^{z - 1} ds = \frac{z^k}{2 (\log 2)^z} = o(1).
    \]
    Hence
    \[
    (c - \delta) \gamma_k(z) + o(1) \le f_z(A) \le
    (c + \delta) \gamma_k(z) + o(1).
    \]
    Let $z$ tend to $0$. We obrain
    \[
    c - \delta \le \underline{\eta}(A) \le \overline{\eta}(A) \le c + \delta.
    \]
    But $\delta$ is arbitrary small. Thus $\eta(A) = c$ and Theorem \ref{pr_dens}D is proved.
    
    Let us prove Theorem \ref{pr_dens}E.
    For an element $b \in B$ let us define $l(b)$ to be the maximum length of the sequence $b_{0}, b_{1}, \ldots, b_{l(b)}$ such that $b_0 = b$ and $b_{i} | b_{i + 1}$ for $0 \le i < l(b)$. 
    Let us assume that $B$ does not contain $n$ elements $b_1, b_2, \ldots, b_n$ such that $b_i | b_{i + 1}$. 
    This means exactly that for each element $b \in B$ we have $l(b) \le n - 2$.
    Let us represent $B$ as a disjoint union $B = \cup_{i = 0}^{n - 2} B_i$, where $B_{i} \vcentcolon = \{ b \in B \, : \, l(b) = i \}$. 
    Then $B_{i}$ are primitive sets. Indeed, if $b, b' \in B_i$ and $b | b'$, then $i = l(b) > l(b') = i$ -- contradiction. 
    
    By Theorem \ref{pr_dens}A
    \[
    n - 1 < \overline{\eta}(B) \le \sum_{i = 0}^{n - 2} \overline{\eta}(B_i) \le  \sum_{i = 0}^{n - 2} 1 = n - 1
    \]
    This gives a contradiction.  
    Hence $B$ contains $n$ elements $b_1, b_2, \ldots, b_n$ such that $b_i | b_{i + 1}$. 
Now we delete these $n$ elements from $B$. The upper primitive density of $B$ remains the same and we can repeat the process of finding other $n$ elements. This finishes the proof of Theorem \ref{pr_dens}E.

    Let us prove Theorem \ref{pr_dens}F. The construction is similar to the proof of Theorem \ref{tightness}.

    Let $\mathbb{P} = \{p_1, p_2, \ldots \}$. Let $C_i$ be a monotonically increasing sequence. 
    Denote $P_i \vcentcolon = \prod_{j < i} p_j$ and let 
    \[
    B_i \vcentcolon = \{ p_i b \, : \, (b, P_i) = 1, \Omega(b) \le C_i  \}.
    \]
    Let $B = \bigcup_{i = 1}^{\infty} B_i$. If $b \in B_i, b' \in B_j$ and $b | b'$, then $j \le i$ and if $j = i$, then $\Omega(b) < \Omega(b') \le C_i$.  Thus $B$ does not contain an infinite sequence $b_1, b_2, \ldots$, such that $b_i | b_{i + 1}$.

    Let us denote $G \vcentcolon = \inf_{z \in (0, 2)} G(z)$. We know that $G > 0$. 
    Theorem \ref{gamma_k(z) main thm} implies that for all $C_i$ large enough we have for each $z \in (0, 1]$
    \[
    f_z(B_i) > \frac{z}{p_i} \prod_{p < p_i} \left( 1 - \frac{z}{p_i} \right)\frac{G C_i}{2}.
    \]
    But
    \[
    \sum_{i} \frac{z}{p_i} \prod_{p < p_i} \left( 1 - \frac{z}{p_i} \right) = d_z(L_{\mathbb{P}}) = 1.
    \]
    Hence if $C_i$ grows fast enough, then
    \[
    f_z(B) = \sum_{i} f_z(B_i) = \infty
    \]
    for all $z \in (0, 1]$. In particular $\eta(B) = \infty$.
\end{proof}

\section{$z$-logarithmic density: Theorem \ref{D_z}} \label{z-density}

To compute $\delta(h_z, A, N)$ we need first to evaluate the sum $\sum_{n \le N} z^{\Omega(n)}/n$.

\begin{lemma} For $0 < z < 2$ \label{z^Omega sum}
\[
\sum_{n \le x} z^{\Omega(n)} = z G(z) x (\log x)^{z - 1} \left(1 + O_z\left( \frac{1}{\log x} \right) \right), 
\]
\[
\sum_{n \le x} 2^{\Omega(n)} \sim C_2 x (\log x)^2,
\]
where $C_2 = (8 \log 2)^{-1} \prod_{p > 2} (1 + 1/p(p-2))$.
\end{lemma}
\begin{proof}
    See \cite[p. 301, Theorem 6.2]{Tenenbaum} and \cite[p. 59, exer. 57]{Tenenbaum}.
\end{proof}

\begin{lemma} \label{dens_whole}
    If $0 < z < 2$, then
    \[
    \sum_{n \le N} \frac{z^{\Omega(n)}}{n} = G(z) (\log N)^z + O_z \left( (\log N)^{z-1} \right),
    \]
    \[
    \sum_{n \le N} \frac{2^{\Omega(n)}}{n} \sim (C_2 / 3) (\log N)^3.
    \]
\end{lemma}

\begin{proof}
    It follows from Lemma \ref{z^Omega sum} and integration by parts.
\end{proof}

\begin{lemma}[Sathe-Selberg] \label{Selb}
For $k \le (2 - \delta) \log \log x$, 
\[
N_{k+1}(x) = G \left( \frac{k}{\log \log x} \right) \frac{x}{\log x} \frac{(\log \log x)^k}{k!} \left( 1 + O_{\delta}\left( \frac{k}{(\log \log x)^2} \right) \right),
\]
\end{lemma}

\begin{proof}
    See \cite{selberg1954note} or \cite[p. 304, Theorem 6.5]{Tenenbaum}.
\end{proof}

\begin{lemma} \label{z log log}
    Let $0 < z < 2$ and $k = [z \log \log N]$. Then
    \[
    \sum_{\substack{n \le N \\ \Omega(n) = k}} \frac{z^k}{n} \sim 
    \frac{G(z) (\log N)^z}{\sqrt{2 \pi z \log \log N}}. 
    \]
\end{lemma}

\begin{proof}
By Lemma \ref{Selb} uniformly in the range $N^{(\log \log N)^{-1/z}} \le x \le N$ we have
\[
N_k(x) \sim G(z) \frac{x}{\log x} \frac{(\log \log x)^{k-1}}{(k-1)!}.
\]
Let $z \log \log N = k + \varepsilon$. We know that $\varepsilon < 1$. Integrating by parts we obtain
\[
\sum_{\substack{N^{(\log \log N)^{-1/z}} < n \le N \\ \Omega(n) = k}} \frac{z^k}{n} \sim \frac{G(z) (z \log \log N)^k}{k!} = G(z) \left( \frac{z \log \log N}{k} \right)^k \frac{k^k}{k!} \sim
\]
\[
G(z) \left( \frac{k + \varepsilon}{k} \right)^k \frac{e^k}{\sqrt{2 \pi k}} \sim \frac{G(z) e^{k + \varepsilon}}{\sqrt{2 \pi k}} \sim \frac{G(z) (\log N)^z}{\sqrt{2 \pi z \log \log N}}.
\]
By Lemma \ref{dens_whole} 
\[
\sum_{\substack{n \le N^{(\log \log N)^{-1/z}} \\ \Omega(n) = k}} \frac{z^k}{n} \le \sum_{n \le N^{(\log \log N)^{-1/z}}} \frac{z^{\Omega(n)}}{n} \ll \frac{G(z) (\log N)^z}{\log \log N} = o\left(\frac{G(z) (\log N)^z}{\sqrt{2 \pi z \log \log N}} \right).
\]
\end{proof}

\begin{lemma} \label{Omega - omega}
    There exists a constant $C = C(z)$ such that
    \[
    \sum_{\substack{m \le N \\ \Omega(m) - \omega(m) \ge C \log x}} \frac{z^{\Omega(m)}}{m} = o \left( \frac{(\log N)^z}{x^{1/2}}\right)
    \]
\end{lemma}

\begin{proof}
    Each number $m$ can be written in the form $m = 2^r m'$, where $m'$ is odd.
    Define the sets $M_1 = \{m \le N : 2^{[(C/2) \log x]}|m \}$, $M_2 = \{ m \le N : \Omega(m') - \omega(m') \ge (C/2) \log x\}$.
    
    Let $m$ be an integer for which $\Omega(m) - \omega(m) \ge C \log x$. Then $m \in M_1 \cup M_2$.

    \[
    \sum_{m \in M_1} \frac{z^{\Omega(m)}}{m} \le (z/2)^{(C/2)\log x - 1} \sum_{n \le N} \frac{z^{\Omega(n)}}{n} \ll x^{(C/2) \log(z/2)} (\log N)^z.
    \]
    Each element $m \in M_2$ is divisible by an odd square $m_2$ such that $\Omega(m_2) - \omega(m_2) \ge (C/8) \log x$. Hence
    \[
    \sum_{m \in M_2} \frac{z^{\Omega(m)}}{m} \le \left( 
    \sum_{p > 2} \frac{z^2}{p^2} \right)^{(C/8)\log x} 
    \sum_{n \le N} \frac{z^{\Omega(n)}}{n} \le \
    \left( \sum_{p > 2} \frac{4}{p^2} \right)^{(C/8)\log x} 
    \sum_{n \le N} \frac{z^{\Omega(n)}}{n} \le
    \]
    \[
    0.9^{(C/8)\log x} \sum_{n \le N} \frac{z^{\Omega(n)}}{n} \ll
    x^{\log(0.9) (C/8)} (\log N)^z.
    \]
    Now we fix $\varepsilon > 0$ and choose $C > 0$ such that $\max(\log(0.9) (C/8), (C/2) \log(z/2)) < -1/2 - \varepsilon$.
\end{proof}

\begin{lemma} \label{Sark}
    If $f$ is a completely multiplicative function such that  there is a prime $p$ with $f(p) > 1$, then there are numbers $C = C(p, f(p)) > 0$, $N_0 = N_0(p)$  that for every $N > N_0$ there exists a primitive $A \subset \{1, 2, \ldots, N \}$ such that
    \[
    \delta(f, A, N) \ge C.
    \]
    One can take $N_0 = p^2$ and
    \[
    C = \frac{f(p) - 1}{f(p) \left( \frac{\log p^3}{\log 2} \right)}.
    \]
\end{lemma}

\begin{proof}
    See \cite[Theorem 4, Proof of Theorem 4]{Sarkozy}.
\end{proof}

\begin{proof}[Proof of Theorem \ref{D_z}A and \ref{D_z}C]
    Theorem \ref{D_z}C follows directly from Lemma \ref{Sark}, because if $z > 2$, then $f(2) = z/2 > 1$.

    Now we prove Theorem \ref{D_z}A following the proof in \cite{Szemeredi}.

    Let $k = [z \log \log N]$.
    
    In view of the Lemma \ref{z log log} it is enough to show that for each primitive set $A \subset \{1, 2, \ldots, N \}$
    \[
    \sum_{a \in A} \frac{z^{\Omega(a)}}{n} \le (1 + o(1)) \sum_{\substack{n \le N \\ \Omega(n) = k}} \frac{z^k}{n}.
    \]
    We can assume that $A$ consists of the elements $a$ such that $\Omega(a) - \omega(a) < C \log k$ as the sum over other elements is small due to Lemma \ref{Omega - omega}.
    
    Let us divide $A$ into three disjoint subsets $A = A' \cup A'' \cup A'''$. 
    \[
    A' = \{a \in A : \Omega(a) > k\}, \quad 
    A'' = \{a \in A : \Omega(a) = k\}, \quad 
    A''' = \{a \in A : \Omega(a) < k\}.
    \]
    Let $r = \max\{\Omega(a) : a \in A'\}$. Let $A_{(r)} = \{ a \in A : \Omega(a) = r\}$. Let $B'_{(r)}$ be the set of all divisors of elements of $A_{(r)}$ with $r-1$ prime divisors counted with multiplicity. Note that $B'_{(r)}$ and $A$ are disjoint.

    We replace $A$ with the new primitive set $(A \setminus A_{(r)}) \cup B'_{(r)}$. We repeat this process until $A' = \varnothing$. Let us denote the resulting $A$ as $A_{new}$ and the $A$ from which we started as $A_{old}$. Let us evaluate, how the sum $\sum_{a \in A} \frac{z^{\Omega(a)}}{n}$ changes during this process.

    \[
    \sum_{a \in B'_{(r)}} \frac{z^{\Omega(a)}}{a} \left( \sum_{p \le N} \frac{z}{p} \right) \ge 
    \sum_{a \in A_{(r)}} \omega(a) \frac{z^{\Omega(a)}}{a} \ge (r - C \log k) \sum_{a \in A_{(r)}} 
    \frac{z^{\Omega(a)}}{a}.
    \]
    Using the second theorem of Mertens $\sum_{p \le N}  p^{-1} < \log \log N + c$, we obtain
    \[
    \sum_{a \in B'_{(r)}} \frac{z^{\Omega(a)}}{a} \ge \frac{r - C \log k}{z \log \log N + c'} \sum_{a \in A_{(r)}} 
    \frac{z^{\Omega(a)}}{a}.
    \]
    If $r > k + 2 C \log k$, then 
    \[
    \frac{r - C \log k}{z \log \log N + c'} > 1
    \]
    And for $r > k$ we have
    \[
    \frac{r - C \log k}{z \log \log N + c'} > 1 - \frac{2C \log k}{k}.
    \]
    By the induction argument 
    \[
    \sum_{a \in A_{new}} \frac{z^{\Omega(a)}}{a} \ge \left( 1 - \frac{2C \log k}{k} \right)^{2 C \log k} \sum_{a \in A_{old}} \frac{z^{\Omega(a)}}{a} = (1 + o(1)) \sum_{a \in A_{old}} \frac{z^{\Omega(a)}}{a}.
    \]

    So we can assume that $A' = \varnothing$. 
    
    Now let $r = \min \{ \Omega(a) : a \in A'''\}$. Let $B'''_{(r)}$ be the set of all numbers of the form $p a_r$, where  $a_r \in A_{(r)}$ and $p < N^{1 / k^2}$. Again $B'''_{(r)}$ and $A$ are disjoint. And we replace $A$ with the new primitive set $(A \setminus A_{(r)}) \cup B'''_{(r)}$. Note that the new $A$ may have elements larger than $N$. We repeat this process until $A''' = \varnothing$. Again we denote the resulting $A$ as $A_{new}$ and the $A$ from which we started as $A_{old}$.

    $A_{new}$ will be a subset of $[1, N^{1 + 1/k}]$. We have
    \[
    \sum_{a \in A_{(r)}} \frac{z^{\Omega(a)}}{a} \left(\sum_{p < N^{1/k^2}} \frac{z}{p} \right) \le (r + 1) \sum_{a \in B'''_{(r)}} \frac{z^{\Omega(a)}}{a}
    \]
    
    Using the second theorem of Mertens we obtain 
    \[
    \sum_{a \in B'''_{(r)}} \frac{z^{\Omega(a)}}{a} \ge \frac{k - 3 \log k}{r + 1} \sum_{a \in A_{(r)}} \frac{z^{\Omega(a)}}{a}.
    \]
    If $r + 1 \le k - 3 \log k$ we have
    \[
    \frac{k - 3 \log k}{r + 1} \ge 1.
    \]
    Since $r + 1 \le k$, we always have 
    \[
    \frac{k - 3 \log k}{r + 1} \ge 1 - \frac{3 \log k}{k}.
    \]
    And we obtain
    \[
    \sum_{a \in A_{new}} \frac{z^{\Omega(a)}}{a} \ge \left( 1 - \frac{3 \log k}{k} \right)^{3 \log k} \sum_{a \in A_{old}} \frac{z^{\Omega(a)}}{a} = (1 + o(1)) \sum_{a \in A_{old}} \frac{z^{\Omega(a)}}{a}.
    \]
    We have
    \[
    \sum_{a \in A_{new}} \frac{z^{\Omega(a)}}{a} \le \sum_{a \in A_{new} \cap [1, N]} \frac{z^{\Omega(a)}}{a} + \sum_{n = N+1}^{N^{1 + 1/k}} \frac{z^{\Omega(n)}}{n}.
    \]
    By Lemma \ref{dens_whole}
    \[
    \sum_{n = N+1}^{N^{1 + 1/k}} \frac{z^{\Omega(n)}}{n} \ll \left(\left(1 + \frac{1}{k} \right)^z - 1\right) (\log N)^z + (\log N)^{z-1} \ll \frac{(\log N)^z}{k} = o\left(\sum_{\substack{n \le N \\ \Omega(n) = k}} \frac{z^k}{n} \right).
    \]
    Now $A_{new} \cap [1, N]$ is a subset of $\{ n \le N : \Omega(n) = k \}$.
    
    Hence for an arbitrary primitive set $A \subset \{1, 2, \ldots, N \}$
    \[
    \sum_{a \in A} \frac{z^{\Omega(a)}}{n} \le (1 + o(1)) \sum_{\substack{n \le N \\ \Omega(n) = k}} \frac{z^k}{n}.
    \]
\end{proof}

\begin{lemma} \label{1/n asymp}
For $k \ge 1$ and $x \ge 2^{k+2}$
    \[
    \sum_{\substack{n \le x \\ \Omega(n) = k}} \frac{2^k}{n} \asymp \sum_{1 \le j \le k} \frac{\left(2 \log \log \frac{x}{2^k}\right)^j}{j!}.
    \]
\end{lemma}
\begin{proof}
    Note that $(2 - \rho) G(\rho) \asymp 1$ in the range $\rho \in [0, 2)$. Hence by Proposition \ref{BDN} uniformly in the range $x \ge 3 \cdot 2^k$
    \[
    N_k(x) \asymp \frac{x}{2^k} \left( \log \frac{x}{2^k} \right)^{-1} \sum_{0 \le j < k} \frac{\left(2 \log \log \frac{x}{2^k} \right)^j}{j!}.
    \]
    Integrating by parts we obtain
    \[
    \sum_{\substack{n \le x \\ \Omega(n) = k}} \frac{2^k}{n} \asymp 1 + \left( \log \frac{x}{2^k} \right)^{-1} \sum_{0 \le j < k} \frac{\left(2 \log \log \frac{x}{2^k} \right)^j}{j!} 
    + \int_{3 \cdot 2^k}^x \sum_{0 \le j < k} \frac{\left(2 \log \log \frac{t}{2^k} \right)^j}{j! \, \, t \log \frac{t}{2^k}} dt.
    \]
    \[
    \int_{3 \cdot 2^k}^x \sum_{0 \le j < k} \frac{\left(2 \log \log \frac{t}{2^k} \right)^j}{j! \, \, t \log \frac{t}{2^k}} dt = 
    \int_3^{x / 2^k} \sum_{0 \le j < k} \frac{\left(2 \log \log r \right)^j}{j! \, \, r \log r} dr =
    \]
    \[
    \int_{\log \log 3}^{\log \log \frac{x}{2^k}} 
    \sum_{0 \le j < k} \frac{(2 y)^j}{j!} \, dy 
    \asymp \sum_{1 \le j \le k} \frac{(2 \log \log \frac{x}{2^k})^j}{j!}.
    \]
\end{proof}

\begin{lemma} \label{poisson tails}
    Let $\alpha \ge 1$ and $t > 0$, then
    \[
    \sum_{0 \le j < \alpha t} \frac{t^j}{j!} \ge e^{t} - e^{t(\alpha - \alpha \log \alpha)}
    \]
\end{lemma}
\begin{proof}
    See, for example, \cite[Proposition 0.3]{Ford}.
\end{proof}

\begin{proof}[Proposition \ref{improvement} implies Theorem \ref{D_z}B]
    \,
    
    By Lemma \ref{dens_whole}
    \[
    D_2(N) \ll (\log N)^{-3} \max_{A \,\, \text{primitive}} \sum_{\substack{n \in A \\ n \le N}} \frac{ 2^{\Omega(a)}}{a}.
    \]
    Hence Proposition \ref{improvement} implies that $D_2(N) \ll (\log N)^{-1}$.

    Now let us prove that $D_2(N) \gg (\log N)^{-1}$. It is enough to find such primitive $A \subset \{1, 2, \ldots, N \}$ that
    \[
    \sum_{a \in A} \frac{2^{\Omega(a)}}{a} \gg (\log N)^2.
    \]
    Let us fix any constant $d > 2$. Take $k = [d \log \log N]$ and $A = \mathbb{P}_k \cap [1, N]$. Then Lemmas \ref{1/n asymp} and \ref{poisson tails} give
    \[
    \sum_{\substack{n \le N \\ \Omega(n) = k}} \frac{2^k}{n} \asymp \sum_{1 \le j \le k} \frac{(2 \log \log \frac{N}{2^k})^j}{j!} \gg_d (\log N)^2.
    \]
    
\end{proof}

\section{Proof of Proposition \ref{improvement} and Theorem \ref{N_k-constant}} \label{Const}

First we prove Proposition \ref{improvement} with explicit constant. For that we need
\begin{lemma} \label{286}
    For $x \ge 286$
    \[
    \prod_{p \le x} (1 - 1/p)^{-1} \le  e^{\gamma} (\log x) \left( 1 + \frac{2}{\log^2 x} \right),
    \]
    here $\gamma$ is the Euler–Mascheroni constant.
\end{lemma}
\begin{proof}
    See \cite[Theorem 8]{approx}.
\end{proof}

\begin{lemma} \label{improvement_const}
Let $A$ be a primitive set and $N \ge 2$. Then
    \[
    \sum_{
    \substack{a \in A \\ P(a) \le N}} \frac{2^{\Omega(a)}}{a} \le  2.486 \, (\log N)^2.
    \]
\end{lemma}

\begin{proof}[Proof of Lemma \ref{improvement_const} and Proposition \ref{improvement}]
    Let $p_1, p_2, \ldots$ be the odd primes in increasing order. For some $r$ we have $p_r \le N < p_{r+1}$.
    Let us set $p_1 \prec p_2 \prec \ldots \prec p_r \prec 2 \prec p_{r+1} \prec p_{r+2} \prec \ldots$. Let $f(n) = \frac{2^{\Omega(n)}}{n}$. Theorem \ref{main} gives
    \[
    \sum_{\substack{a \in A \\ P(a) \le N \\ a \,\, \text{odd}}} \frac{2^{\Omega(a)}}{a} \prod_{2 < p < P(a)} (1 - 2/p) + \sum_{\substack{a \in A \\ P(a) < N \\ a \,\, \text{even}}} \frac{2^{\Omega(a)}}{a} \prod_{2 < p \le N} (1 - 2/p) \le 1.
    \]
    Hence
    \[
    \sum_{\substack{a \in A \\ P(a) \le N}} \frac{2^{\Omega(a)}}{a} \le \prod_{2 < p \le N} (1 - 2/p)^{-1} \ll (\log N)^2.
    \]
    This proves Proposition \ref{improvement}.

    Now assume that $N \ge 286$. Then by Lemma \ref{286}
    \[
    \prod_{2 < p \le N} (1 - 2/p)^{-1} = (1/4) \prod_{p \le N} (1 - 1/p)^{-2} \prod_{2 < p \le N} \frac{(1 - 1/p)^2}{1 - 2/p} \le 
    \]
    \[
    (1/4) e^{2 \gamma}  \left( 1 + \frac{2}{\log^2 286} \right)^2 \prod_{p > 2} \left( 1 - \frac{1}{(p-1)^2} \right)^{-1} (\log N)^2 \le 2.486 (\log N)^2.
    \]
    For smaller $N$ the inequality $\prod_{2 < p \le N} (1 - 2/p)^{-1} \le 2.486 (\log N)^2$ can be checked explicitly. In particular $2.486 (\log 2)^2 > 1$.
\end{proof}

    \begin{lemma} \label{c}
       For $Z, A > 0$ denote 
       \[
       G(Z, A) \vcentcolon = \sum_{\substack{n \le Z \\ \omega(n) \ge A}} 1.
       \]
       Then 
       \[
       G(Z, A) \le c \, 2^{-A} Z \log(Z + 2).
       \]
       One can take $c = 1.123$.
    \end{lemma}
    \begin{proof}
    
        This is \cite[Lemma 1]{Elem_upperbound}.

        Denote $d(n) \vcentcolon = \sum_{d | n} 1$.
        If $\omega(n) \ge A$, then $d(n) \ge 2^A$ and hence
        \[
        2^{A} G(Z, A) \le \sum_{n \le Z} d(n).
        \]
        \cite[Theorem 1.1]{Dir_bound} gives for $Z \ge 2$
        \[
        (Z \log Z)^{-1} \sum_{n \le Z} d(n) \le 1 + \frac{2 \gamma - 1}{\log Z} + \frac{1}{\sqrt{Z} \log Z}.
        \]
        This implies for $Z \ge 30$ that
        \[
        \sum_{n \le Z} d(n) \le 1.123 \, Z \log(Z + 2),
        \]
        and for smaller $Z$ this can be checked explicitly.
        
    \end{proof}

    Let us denote
    \[
    F(N, k) \vcentcolon = \sum_{\substack{\Omega(n) = k \\ P(n) \le N}} \frac{1}{n}.
    \]

    \begin{lemma} \label{copy}
    \[
    N_k(x) \le \frac{x}{2^k} + c \, x \log(x+2)  \sum_{j = 0}^k 2^{j-k} F(2^k, j).
    \]
    \end{lemma}
    \begin{proof}
        See \cite[Section 4]{Elem_upperbound}. Let us give an outline of the proof for the sake of completeness.
        
        Denote $S \vcentcolon = \mathbb{P}_k \cap [1, x]$. Let $S_1$ denote the set of positive integers $n$ for which $n \le x$ and there exist a positive integer $t$ such that $t > 2^k$ and $t^2 | n$. Let $S_2 \vcentcolon = S \setminus S_1$. Obviously $N_k(x) = |S| \le |S_1| + |S_2|$.

        We have
        \[
        |S_1| \le \sum_{t = 2^i + 1}^{\infty} \left[ \frac{x}{t^2} \right] < \frac{x}{2^i}.
        \]

        Each element $n \in S_2$ can be written in the form $n = n_1 n_2$, where $P(n_i) \le 2^k$, $p(n_2) > 2^k$. Then $n_2$ is square-free since $n \notin S_1$. Thus $\omega(n_2) = k - \Omega(n_1)$. We have
        \[
        |S_2| = \sum_{j = 0}^k \sum_{\substack{n_1 \le x \\ P(n_1) \le 2^k \\ \Omega(n_1) = j}}
        \sum_{\substack{n_2 \le x/n_1 \\ p(n_2) > 2^k \\ \omega(n_2) = k - j}} 1 \le 
        \sum_{j = 0}^k \sum_{\substack{n_1 \le x \\ P(n_1) \le 2^k \\ \Omega(n_1) = j}}
        \sum_{\substack{n_2 \le x/n_1 \\ \omega(n_2) = k - j}} 1.
        \]
        To estimate the inner sum we use Lemma \ref{c} with $Z = x/n_1$, $A = k - j$.
        We obtain
        \[
        |S_2| \le \sum_{j = 0}^{k} \sum_{\substack{n_1 \le x \\ P(n_1) \le 2^k \\ \Omega(n_1) = j}}
        c \, 2^{j - k} \frac{x}{n_1} \log \left( \frac{x}{n_1} + 2 \right) \le 
        c \, x \log(x+2)  \sum_{j = 0}^k 2^{j-k} F(2^k, j).
        \]
    \end{proof}

    \begin{proof}[Proof of Theorem \ref{N_k-constant}]
    Lemma \ref{improvement_const} gives $F(N, j) \le 2.486 \,\, 2^{-j} (\log N)^2$.
    Then Lemmas \ref{copy} and \ref{c} give us
        \[
        N_k(x) \le \frac{x}{2^k} + \frac{2.486 \, c \, x \log(x + 2)}{2^k} \sum_{j = 0}^{k} (\log 2^{k})^2 \le
        \frac{x}{2^k} + 1.342 \frac{k^3 x \log(x + 2)}{2^k}.
        \]
        If $k \ge 10$, then we can assume, that $x \ge 2^k$, otherwise $N_k(x) = 0$. In this case
        \[
        1 + 1.342 \, k^3 \log(x + 2) \le 1.35 \, k^3 \log(x).
        \]
        This proves Theorem \ref{N_k-constant} in case $k \ge 10$.
        
        If $2 \le k \le 9$, then $k^3 / 2^k \ge 1$ and Theorem \ref{N_k-constant} is trivial.

        Finally, in case $k = 1$ the inequality \ref{le ?} is trivial if $\log x > 2$ and for smaller $x$ it can be easily checked.
    \end{proof}

\textbf{Acknowledgements.} I am grateful to my research advisor A. Kalmynin for valuable discussions.

\bibliography{pr1}
\bibliographystyle{plainurl}

\end{document}